\documentclass[a4paper, reqno, 14pt]{amsart}

\usepackage[usenames,dvipsnames]{color}
\usepackage{amsthm,amsfonts,amssymb,amsmath,amsxtra}
\usepackage{tikz}
\usepackage{tkz-euclide}
\usepackage[all]{xy}
\SelectTips{cm}{}
\usepackage{xr-hyper}
\usepackage[colorlinks=
   citecolor=Black,
   linkcolor=Red,
   urlcolor=Blue]{hyperref}
\usepackage{verbatim}

\usepackage[margin=1.25in]{geometry}
\usepackage{mathrsfs}

\RequirePackage{xspace}
\RequirePackage{etoolbox}
\RequirePackage{varwidth}
\RequirePackage{enumitem}
\RequirePackage{tensor}
\RequirePackage{mathtools}
\RequirePackage{longtable}
\RequirePackage{multirow}

\newcommand{\ui}{\underline{i}}
\def\JJ{\mathbb J}

\def\ge{\geqslant}
\def\le{\leqslant}
\def\a{\alpha}

\def\g{\gamma}
\def\G{\Gamma}
\def\d{\delta}

\def\o{\omega}

\def\s{\sigma}
\def\t{\tau}
\def\th{\theta}
\def\k{\kappa}
\def\l{\lambda}

\def\i{^{-1}}

\def\<{\langle}
\def\>{\rangle}

\newcommand{\fkP}{\ensuremath{\mathfrak{P}}\xspace}

\newcommand{\bE}{\mathbf E}

\newcommand{\bK}{\mathbf K}

\def\brF{\breve F}

\def\brI{\breve \CI}
\def\brP{\breve \CP}
\def\brK{\breve \CK}

\def\tSS{\tilde{\mathbb S}}

\def\aa{\mathbf a}

\newcommand{\BA}{\ensuremath{\mathbb {A}}\xspace}

\newcommand{\BD}{\ensuremath{\mathbb {D}}\xspace}

\newcommand{\BF}{\ensuremath{\mathbb {F}}\xspace}
\newcommand{{\BG}}{\ensuremath{\mathbb {G}}\xspace}

\newcommand{\BJ}{\ensuremath{\mathbb {J}}\xspace}
\newcommand{{\BK}}{\ensuremath{\mathbb {K}}\xspace}

\newcommand{\BN}{\ensuremath{\mathbb {N}}\xspace}

\newcommand{\BQ}{\ensuremath{\mathbb {Q}}\xspace}
\newcommand{\BR}{\ensuremath{\mathbb {R}}\xspace}
\newcommand{\BS}{\ensuremath{\mathbb {S}}\xspace}

\newcommand{\BZ}{\ensuremath{\mathbb {Z}}\xspace}

\newcommand{\CA}{\ensuremath{\mathcal {A}}\xspace}

\newcommand{\CF}{\ensuremath{\mathcal {F}}\xspace}

\newcommand{\CI}{\ensuremath{\mathcal {I}}\xspace}

\newcommand{\CK}{\ensuremath{\mathcal {K}}\xspace}

\newcommand{\CM}{\ensuremath{\mathcal {M}}\xspace}

\newcommand{\CO}{\ensuremath{\mathcal {O}}\xspace}
\newcommand{\CP}{\ensuremath{\mathcal {P}}\xspace}

\newcommand{\Ad}{{\mathrm{Ad}}}
\newcommand{\ad}{{\mathrm{ad}}}

\DeclareMathOperator{\dist}{dist}

\DeclareMathOperator{\Adm}{Adm}
\newcommand{\EO}[1]{\Adm(\{\mu\})\cap {}^{#1}\tW}

\DeclareMathOperator{\Gal}{Gal}
\newcommand{\Gad}{G^{\ad}}

\newcommand{\id}{\ensuremath{\mathrm{id}}\xspace}
\let\Im\relax
\DeclareMathOperator{\Im}{Im}

\newcommand{\Int}{\ensuremath{\mathrm{Int}}\xspace}

\DeclareMathOperator{\Res}{Res}

\DeclareMathOperator{\Spec}{Spec}

\def\tw{\tilde w}
\def\tW{\tilde W}
\def\kk{\mathbf k}
\def\ua{{\underline{\alpha}}}
\def\uK{{\underline{K}}}

\DeclareMathOperator{\supp}{supp}
\newtheorem{theorem}{Theorem}
\newtheorem{theoremA}{Theorem}

\newtheorem{proposition}[theorem]{Proposition}
\newtheorem{lemma}[theorem]{Lemma}

\newtheorem{corollary}[theorem]{Corollary}

\theoremstyle{definition}
\newtheorem{definition}[theorem]{Definition}

\newtheorem{remark}[theorem]{Remark}

\numberwithin{equation}{section}
\numberwithin{theorem}{section}

\setitemize[0]{leftmargin=*,itemsep=\the\smallskipamount}
\setenumerate[0]{leftmargin=*,itemsep=\the\smallskipamount}

\renewcommand{\to}{%
   \ifbool{@display}{\longrightarrow}{\rightarrow}%
   }
\let\shortmapsto\mapsto
\renewcommand{\mapsto}{%
   \ifbool{@display}{\longmapsto}{\shortmapsto}%
   }
\newcommand{\isoarrow}{%
   \ifbool{@display}{\overset{\sim}{\longrightarrow}}{\xrightarrow\sim}%
   }

\begin{document}
\title{Fully Hodge-Newton decomposable Shimura varieties}
\author[U.~G\"{o}rtz]{Ulrich G\"{o}rtz}
\address{Ulrich G\"{o}rtz\\Fakult\"at f\"ur Mathematik/Institut f\"ur Experimentelle Mathematik\\Universit\"at Duisburg-Essen\\45117 Essen\\Germany}
\email{ulrich.goertz@uni-due.de}
\author[X.~He]{Xuhua He}
\address{Xuhua He, Department of Mathematics, University of Maryland, College Park, MD 20742, USA}
\email{xuhuahe@math.umd.edu}
\thanks{U.~G.~was partially supported by DFG Transregio-Sonderforschungsbereich 45. X.~H.~was partially supported by NSF DMS-1463852. S.~N.~was partially supported by NSFC (No. 11501547 and No. 11321101.) and QYZDB-SSW-SYS007.}
\author[S.~Nie]{Sian Nie}
\address{Sian Nie, Institute of Mathematics, Academy of Mathematics and Systems Science, Chinese Academy of Sciences, 100190, Beijing, China}
\email{niesian@amss.ac.cn}

\keywords{Reduction of Shimura varieties; Affine Deligne-Lusztig varieties;
Newton stratification; Hodge-Newton decomposition}
\subjclass[2010]{11G18, 14G35, 20G25}
\begin{abstract}
The motivation for this paper is the study of arithmetic properties of Shimura varieties, in particular the Newton stratification of the special fiber of a suitable integral model at a prime with parahoric level structure. This is closely related to the structure of Rapoport-Zink spaces and of affine Deligne-Lusztig varieties.

We prove a Hodge-Newton decomposition for affine Deligne-Lusztig varieties and for the special fibres of Rapoport-Zink spaces, relating these spaces to analogous ones defined in terms of Levi subgroups, under a certain condition (Hodge-Newton decomposability) which can be phrased in combinatorial terms.

Second, we study the Shimura varieties in which every non-basic $\s$-isogeny class is Hodge-Newton decomposable. We show that (assuming the axioms of \cite{HR}) this condition is equivalent to nice conditions on either the basic locus, or on all the non-basic Newton strata of the Shimura varieties. We also give a complete classification of Shimura varieties satisfying these conditions.

While previous results along these lines often have restrictions to hyperspecial (or at least maximal parahoric) level structure, and/or quasi-split underlying group, we handle the cases of arbitrary parahoric level structure, and of possibly non-quasi-split underlying groups. This results in a large number of new cases of Shimura varieties where a simple description of the basic locus can be expected. As a striking consequence of the results, we obtain that this property is independent of the parahoric subgroup chosen as level structure.

We expect that our conditions are closely related to the question whether the weakly admissible and admissible loci coincide.
\end{abstract}

\maketitle

\section{Introduction}
\subsection{Motivation}
Understanding arithmetic properties of Shimura varieties has been and still is a key problem in number theory and arithmetic geometry. One important tool (for Shimura varieties of PEL type, where points correspond to abelian varieties with additional structure) is to study the Newton stratification of the special fibre of a suitable integral model of the Shimura variety. The Newton strata are the loci where the isogeny class of the $p$-divisible group of the corresponding abelian variety is constant. There is a unique closed Newton stratum, the so-called basic locus. In the Siegel case, this is the supersingular locus. A similar picture is available in the more general case of Shimura varieties of Hodge type.

The starting point for this paper is the observation that in certain cases, the basic locus admits a nice description as a union of classical Deligne-Lusztig varieties, the index set and the closure relations between the strata being encoded in a Bruhat-Tits building attached to the group theoretic data coming with the Shimura variety; see Section~\ref{subsec:interpretation} for references. While a simple description like this is not available for general Shimura varieties, many of the Shimura varieties which have been used, with great success, towards applications in the realm of the Langlands program, for instance those used by Harris and Taylor in their proof of the local Langlands correspondence for the general linear group, belong to this special class. Precise information of this kind about the basic locus has proved to be useful for several purposes: For example, to compute intersection numbers of special cycles, as in the Kudla-Rapoport program (see e.g., \cite{kudla-rapoport}, \cite{kudla-rapoport-2}) or in work towards Zhang's Arithmetic Fundamental Lemma (e.g., \cite{Rapoport-Terstiege-Zhang}). Helm, Tian and Xiao use cycles found in the supersingular locus to prove the Tate conjecture for certain Shimura varieties \cite{Tian-Xiao-2}, \cite{Helm-Tian-Xiao}.

The non-basic strata, on the other hand, can sometimes be related to Levi
subgroups of the underlying algebraic group. Those strata are called
\emph{Hodge-Newton decomposable}, a condition which has been studied before in
the context of unramified (or at least quasi-split) groups. One can hope to
study those strata (and their cohomology, etc.) by a kind of induction process,
see e.g.~the work of Boyer~\cite{Boyer} and of Shen~\cite{Shen}.

It turns out that there is a close relationship between these two phenomena, made precise in Theorems~\ref{thmB} and~\ref{thmC} below. Apart from this connection, these theorems also show that the existence of a ``nice description'' of the basic locus, as made precise below and in Section~\ref{subsec:4-implies-5}, is equivalent to a very simple condition on the coweight $\mu$ coming from the Shimura datum (we call $\mu$ \emph{minute}, Def.~\ref{def:minute}) and is also equivalent to a certain condition which involves only the non-basic Newton strata.

In the theorems below, we always include the case of a general parahoric
subgroup, even though this adds considerably to the difficulty of the proofs.
Since in this way we obtain a large number of cases which have not been treated
so far in the context of Shimura varieties, we think that this additional
effort is justified. Furthermore, and this is quite a striking point which we
would like to emphasize, it turns out that the existence of a simple
description of the basic locus is independent of the parahoric. We do not see
any reason why this independence of the parahoric could be expected a priori,
but it is an interesting parallel with the question when the weakly admissible
and admissible loci in the rigid analytic period domain coincide; compare Theorem~\ref{conj-adm-weaklyadm}, proved by Chen, Fargues and Shen, which states that this is equivalent to the conditions in Theorem~\ref{thmB}.

To give a more precise statement, we introduce some notation. We will start in a purely group theoretic setting which does not have to arise from a Shimura datum. Afterwards, we will state a version of the main result for Shimura varieties. In order not to overcharge the introduction, we will not repeat all definitions of ``standard'' notation; the details can be found in Section~\ref{sec:group-theoretic-setup}.

Let $G$ be a connected reductive group over a non-archimedean local field $F$.
Denote by $\s$ the Frobenius of $\brF$, the completion of the maximal
unramified extension of $F$, over $F$. We fix a conjugacy class $\{\mu\}$ of
cocharacters of $G$ over the separable closure $\overline{F}$, and, as usual,
denote by $B(G, \{\mu\})$ the set of ``neutral acceptable'' $\s$-conjugacy
classes in $G(\brF)$ with respect to $\{\mu\}$, and by $\Adm(\{\mu\})$ the
$\{\mu\}$-admissible set, a finite subset of the Iwahori-Weyl group $\tW$.

Denote by $b_0\in G(\brF)$ a representative of the unique basic $\s$-conjugacy class in $B(G, \{\mu\})$.

For a $\s$-stable standard parahoric subgroup $\brK$ (the ``level structure''), which we encode as a subset $K\subseteq \tSS$ of the set of simple affine reflections in $\tW$, and any $b\in G(\brF)$, we have the generalized affine Deligne-Lusztig variety
\[
X(\mu, b)_K = \{ g\in G(\brF)/\brK;\ g\i b\s(g) \in \bigcup_{w\in\Adm(\{\mu\})} \brK w\brK \},
\]
one of the main players in this article. It was introduced by Rapoport in~\cite{rapoport:guide}. This is a subscheme, locally of finite type, of the partial affine flag variety attached to $\brK$ (in the usual sense in equal characteristic; in the sense of Zhu~\cite{Zhu} in mixed characteristic). We can then characterize $B(G, \{\mu\})$ as the set of those $\s$-conjugacy classes $[b]$ such that $X(\mu, b)_K \ne\emptyset$ (the latter property being independent of the choice of representative inside a $\s$-conjugacy class, and of $K$).

\subsection{Hodge-Newton decomposition}
For non-basic $b$, one can sometimes relate $X(\mu, b)_K$ to a Levi subgroup, namely when the pair $(\mu, b)$ is \emph{Hodge-Newton decomposable (HN-decomposable)} (Def.~\ref{def:HN-decomp}). Previously, this notion has been studied in the case of hyperspecial level structure, or at least assuming that $G$ is quasi-split, see~\cite[Def.~4.28]{RV}, \cite[2.5]{CKV}). Our first main result is
\begin{theoremA} (Theorem~\ref{HN-dec})
Suppose that $(\mu, b)$ is HN-decomposable with respect to the Levi subgroup $M$.
Then
\[
X(\mu, b)_{K} \cong \bigsqcup_{P'=M'N'} X^{M'}(\mu_{P'}, b_{P'})_{K_{M'}},
\]
where $P'$ runs through a certain finite set of semistandard parabolic subgroups, $M'$ is the unique Levi subgroup of $P'$ containing the fixed maximal torus, and $X^{M'}(\mu_{P'}, b_{P'})_{K_{M'}}$ is a generalized affine Deligne-Lusztig variety attached to $M'$.
The subsets in the union are open and closed.
\end{theoremA}
See Section~\ref{sec:HN} for the notation used here. The index set of the disjoint union, in the notation of Section~\ref{sec:HN}, is $\mathfrak P^\s/W_K^\s$; see in particular Sections~\ref{HN-setup} and~\ref{HN-setup-2}. In the simple case where the set $\BS_0$ of simple finite reflections is preserved by the $\s$-action on $\tW$ and $K=\BS_0$, this index set has only one element and the isomorphism takes the form $X(\mu, b)_K \cong X^M(\mu_P, b_P)_{K_M}$.

The Hodge-Newton decomposition for $F$-crystals (with additional structures)
originates in Katz's work \cite{Katz}. Its variations for affine
Deligne-Lusztig varieties and Rapoport-Zink spaces with hyperspecial level
structure have been considered by several authors; we mention Kottwitz
\cite{kottwitz:hn}, Mantovan \cite{mantovan}, Shen \cite{Shen},
Chen-Kisin-Viehmann \cite[Prop.~2.5.4, Thm.~2.5.6]{CKV}, and Hong~\cite{hong},
\cite{hong2}. See also \cite[Theorem 2.1.4]{GHKR} and \cite[Proposition 2.5.1
\& Theorem 3.3.1]{GHN}. To the best of our knowledge, the HN decomposition in
the case of a general parahoric subgroup is a new result. If the triple $(G,
\mu, b)$ corresponds to a Rapoport-Zink space, then our result implies an HN
decomposition for the underlying reduced scheme of the RZ space (see
Section~\ref{subsec:HN-for-RZ}). On the other hand, several of the above
mentioned articles establish decompositions for the generic fibre of the RZ
space (in most cases, only hyperspecial level structure is considered,
however). We expect that a HN decomposition of the generic fiber of the RZ
spaces exists, when $(\mu, b)$ is HN decomposable.

\subsection{Fully Hodge-Newton decomposable pairs $(G, \{\mu\})$} We say that a pair $(G, \{\mu\})$ is {\it fully Hodge-Newton decomposable} if every non-basic $\s$-conjugacy class $[b]$ is Hodge-Newton decomposable with respect to some proper standard Levi.

Our second main result is

\begin{theoremA} (Theorem~\ref{main})\label{thmB}
Let $(G, \{\mu\}, \brK)$ be as in~\ref{sec:group-theoretic-setup}. Suppose that $G$ is quasi-simple over $F$. Then the following conditions are equivalent:
\begin{enumerate}[label=(\arabic*)]
\item (HN-decomp)
The pair $(G, \{\mu\})$ is fully Hodge-Newton decomposable;
\item (minute)
The coweight $\mu$ is minute;
\item (non-basic -- leaves)
For any non-basic $[b]$ in $B(G, \{\mu\})$, $\dim X(\mu, b)_K=0$;
\item (EKOR $\subseteq$ Newton)
For any $w \in \Adm(\{\mu\}) \cap {}^K \tW$, there exists a unique $\s$-conjugacy class $[b] \in B(G, \{\mu\})$ such that $X_w(b) \neq \emptyset$;
\item (basic -- EKOR)
The space $X(\mu, b_0)_K$ is naturally a union of classical Deligne-Lusztig varieties.
\end{enumerate}
In particular, the validity of the conditions (3)--(5) only depends on the choice of $(G, \{\mu\})$ and is independent of $K$.
\end{theoremA}
The condition (2) that $\mu$ is minute is an easy to check combinatorial condition, Def.~\ref{def:minute}, even though the statement is a little subtle in the non-quasi-split case. For the meaning of condition (5) (where we should more precisely speak of the perfections of classical Deligne-Lusztig varieties in the case $\mathop{\rm char}(F) =0$) see Sections~\ref{subsec:main-thm} and~\ref{subsec:4-implies-5}.

\subsection{Weakly admissible versus admissible locus}
Now assume that $F$ has characteristic $0$, and that $\mu$ is minuscule. To the triple $(G, b, \{\mu\})$, there is attached a period domain $\CF^{\rm wa}(G, b, \{\mu\})$, the so-called \emph{weakly admissible} locus inside the partial flag variety for $G$ and $\mu$. This is an adic space which contains as an open subspace the \emph{admissible} locus $\CF^{\rm a}$. Weak admissibility is a group theoretic condition which is (in comparison) easy to understand. On the other hand, admissibility is defined in terms of modifications of vector bundles on the Fargues-Fontaine curve and thus intimately linked to Galois representations, and we view it as hard to understand. See~\cite{Rapoport:PD} for a survey and for further references.

In some rare cases, the admissible locus and the weakly admissible locus coincide. Hartl~\cite[Thm.~9.3]{Hartl}, computed all cases where this happens for $G=GL_n$. After learning about Hartl's results, Fargues and Viehmann remarked that for other groups this situation should be characterized by a version of the Hodge-Newton indecomposable condition. (The precise formulation of this condition in the unramified case was certainly known to experts at that time.) The question was further discussed by Rapoport and Fargues. Fargues~\cite{Fargues:private-comm} checked a case for an orthogonal group and conjectured that the two loci coincide if and only if the difference between the Newton polygons of $\mu$ and $b$ is ``very small''. The condition that $(\mu, b)$ is HN-decomposable for all non-basic $b$ is exactly of this type. Equivalently, the condition of $\mu$ being minute makes this precise in a group theoretic way without referring to the set $B(G, \{\mu\})$.

The proofs of the results of Hartl and Fargues show that condition (7) is closely related to the question of HN-decomposability. The idea that the coincidence of admissibility and weak admissibility may be equivalent to the conditions in Theorem~\ref{thmB} was suggested by Rapoport, and can be stated in full generality (without assuming, for instance, that $G$ is unramified or quasi-split) within the setup of the paper at hand.

The following theorem which we stated as a conjecture by Fargues, Rapoport and Viehmann in the first version of this article, was proved by Chen, Fargues and Shen a few months afterwards:

\begin{theorem}\label{conj-adm-weaklyadm} (Chen, Fargues, Shen~\cite{Chen-Fargues-Shen})
If $\mathop{\rm char}(F)=0$ and $\mu$ is minuscule then the conditions in the above theorem are equivalent to
\begin{enumerate}
    \item[(7)] (wa$=$a) The admissible locus $\CF^{\rm a}(G, b_0, \{\mu\})$ is equal to the weakly admissible\\\hspace*{1.2cm} locus $\CF^{\rm wa}(G, b_0, \{\mu\})$.
\end{enumerate}
\end{theorem}

\subsection{Fully Hodge-Newton decomposable Shimura varieties}
If $(G, \{\mu\}, \brK)$ arises from a Shimura datum (assumed to satisfy the axioms of~\cite{HR}), then we may reformulate the above result as follows. These axioms ensure, in particular, that we have a suitable integral model, and good notions of Newton and EKOR strata inside the special fibre. See Section~\ref{shimura-setup} for further details and for the notation used below.

\begin{theoremA} (Theorem~\ref{main'})\label{thmC}
If $(G, \{\mu\}, \brK)$ arises from a Shimura datum as in Section~\ref{shimura-setup}, and $G$ is quasi-simple over $F$, then the conditions (1) to (5) in the previous theorem (more precisely: Theorem~\ref{main}) are equivalent to the following equivalent conditions:
\begin{enumerate}[label=(\arabic*\,$'$), start=3]
\item (non-basic -- leaves\,$'$)
All non-basic Newton strata consist of only finitely many leaves, i.e., in the notation of \cite{HR}, as recalled in Section~\ref{shimura-setup}: For all non-basic $[b]\in B(G, \{\mu\})$,
\[
\Upsilon_K(\delta_K^{-1}([b])) \text{\ is a finite set;}
\]
\item (EKOR $\subseteq$ Newton\,$'$)
Any EKOR stratum of $Sh_K$ is contained in a single Newton stratum;
\item (basic -- EKOR\,$'$)
The basic locus $\mathit S_{K, [b_0]}$ is of EKOR type (Def.~\ref{def:EKORtype}).
\end{enumerate}
In particular, the validity of these conditions only depends on the choice of $(G, \{\mu\})$ and is independent of $K$.
\end{theoremA}

To prove this theorem, we prove that each condition of (3), (4), (5) above is equivalent to its variant (3$'$), (4$'$), (5$'$), respectively, without using Theorem~\ref{main}. See Section~\ref{sec:shimura}.

\subsection{Classification}
One can completely classify those pairs $(G, \{\mu\})$ where the above equivalent conditions are satisfied. Here for simplicity we restrict to the case where $G$ is quasi-simple over $\brF$. See Section~\ref{subsec:interpretation} for the general case.

\begin{theoremA} (Theorem~\ref{triple})
Assume that $G$ is quasi-simple over $\brF$ and $\mu \neq 0$. Then $(G, \{\mu\})$ is fully Hodge-Newton decomposable if and only if the associated triple $(W_a, \mu, \s)$ is one of the following:

\renewcommand{\arraystretch}{1.7}
\[\begin{tabular}{|c|c|c|}
\hline
$(\tilde A_{n-1}, \o^\vee_1, \id)$ & $(\tilde A_{n-1}, \o^\vee_1, \t_1^{n-1})$ & $(\tilde A_{n-1}, \o^\vee_1, \varsigma_0)$
\\
\hline
$(\tilde A_{2m-1}, \o^\vee_1, \t_1 \varsigma_0)$ & $(\tilde A_{n-1}, \o^\vee_1+\o^\vee_{n-1}, \id)$ & $(\tilde A_3, \o^\vee_2, \id)$
\\
\hline
$(\tilde A_3, \o^\vee_2, \varsigma_0)$ & $(\tilde A_3, \o^\vee_2, \t_2)$ &
\\
\hline
\hline
$(\tilde B_n, \o^\vee_1, \id)$ & $(\tilde B_n, \o^\vee_1, \t_1)$ &
\\
\hline\hline
$(\tilde C_n, \o^\vee_1, \id)$ & $(\tilde C_2, \o^\vee_2, \id)$ & $(\tilde C_2, \o^\vee_2, \t_2)$
\\
\hline\hline
$(\tilde D_n, \o^\vee_1, \id)$ &
$(\tilde D_n, \o^\vee_1, \varsigma_0)$ & %  $(\tilde D_n, \o^\vee_1, \t_1 \varsigma_0)$
\\
\hline
\end{tabular}
\]
\end{theoremA}

In many of the cases of the above list which correspond to Shimura varieties (most of them do), the existence of a decomposition of the basic locus into Deligne-Lusztig varieties has already been proved before for hyperspecial level structure or selected maximal parahoric level structures. Mostly, the proofs rely on a direct analysis of the set of lattices that describes the underlying space via Dieudonn\'e theory, rather than the connection with affine Deligne-Lusztig varieties. It seems that for arbitrary level structure this direct analysis is so complicated that it could be carried out only in small rank cases. A different approach is used by Helm, Tian and Xiao~\cite{Tian-Xiao}, \cite{Tian-Xiao-2}, \cite{Helm-Tian-Xiao}; they work directly with the Shimura varieties, and make use of algebraic correspondences between special fibers of Shimura varieties for different groups (notably unitary groups with different signatures). See Section~\ref{subsec:interpretation} for an overview.

\begin{remark}
Compare the similar classification in~\cite{GH}. The list given there classifies those cases where $K$ is a maximal proper subset of $\tSS$ (with $K=\s(K)$) and where it is required that all $\tw\in \Adm(\{\mu\})\cap {}^K\tW$ with $W_{\supp_\s(\tw)}$ finite are $\s$-Coxeter elements in this finite Weyl group. The case where $G$ is not quasi-simple over $\brF$ was not dealt with.
\end{remark}

\subsection{Leitfaden}\label{leitfaden}

\[
\xymatrix@=2cm{
\text{Classif.} \ar@<.7ex>[r]^{\S~\ref{subsec:class-implies-2}} & \ar@<.7ex>[l]^{\S~\ref{subsec:2-implies-class}} \ar[d]^{\S~\ref{subsec:2-implies-3}} (2) \ar@<.7ex>[dl]^{\S~\ref{subsec:2-implies-1}} &  & \ar@{~>}[ll]^{\S~\ref{sec:fc-implies-class}} \text{(6)} & \\
(1)\ar@<.7ex>[ur]^{\S~\ref{subsec:1-implies-2}} & (3)\ar[d]\ar[r]^{\S~\ref{subsec:3-implies-4}} & (4)\ar[d]\ar[r]^{\S~\ref{subsec:4-implies-5}} & (5)\ar[d] \ar@{~>}[u]^{\S~\ref{5-to-FC}}& \\
& (3')\ar[u]^{\S~\ref{subsec:3-equiv-3S}}  &   (4')\ar[u]^{\S~\ref{subsec:4-equiv-4S}} & (5') \ar[u]^{\S~\ref{subsec:5-equiv-5S}}
}
\]

Here, the equivalence between conditions (1) and (2) and the classification are rather easy to prove, so these three characterizations should be seen as one block in the above diagram. The most difficult part of our proof is to show that the Fix Point Condition (6) implies that $\mu$ is minute, for a general parahoric. The case of a maximal parahoric is considerably simpler. In fact, for a maximal parahoric corresponding to a special vertex, we can essentially give a proof which does not use the classification of Dynkin diagrams (see Cor.~\ref{special}). Several shortcuts could be taken: The implication (5) $\Rightarrow$ (4) can be proved directly. From (5), one can obtain the classification in a way similar to~\cite{GH}.

\subsection{Acknowledgments}
We thank Michael Rapoport for interesting discussions and in particular for
bringing up the question how our results relate to the question when the
admissible and weakly admissible locus coincide. We also thank Laurent Fargues
and Haifeng Wu for fruitful conversations and comments.

\section{Notation and preliminaries}

In this section, we describe our setup and recall several definitions and results from other sources. No claim to originality is made.

\subsection{The group-theoretic set-up}
\label{sec:group-theoretic-setup}

Let $F$ be a non-archimedean local field with valuation ring $\CO_F$ and residue field $\BF_q$. Let $\overline{F}$ be an separable closure of $F$, and let $\brF$ be the completion of the maximal unramified extension of $F$ with valuation ring $\CO_{\brF}$ and residue field $\kk=\overline{\BF}_q$. We denote by $\s$ the Frobenius of $\brF$ over $F$, by $\Gamma = \Gal(\overline{F}/F)$ the absolute Galois group of $F$, and by $\G_0 = \Gal(\overline{F}/F^{\rm un})\subseteq \G$ the inertia subgroup. We identify $\G_0$ with the absolute Galois group of $\brF$.

Let $G$ be a connected reductive group over $F$ and $\s$ be the Frobenius automorphism on $G(\brF)$. Let $A$ be a maximal $F$-split torus of $G$, and let $S$ be a maximal $\brF$-split torus of $G$ over $F$ which contains $A$ (for the existence of such a torus see \cite[5.1.12]{BT2}), and let $T$ be its centralizer. Since $G$ is quasi-split over $\brF$, $T$ is a maximal torus. Denote by $N$ the normalizer of $T$ in $G$.

We have the (relative) finite Weyl group $W_0 = N(\brF)/T(\brF)$ and the Iwahori-Weyl group $\tW = N(\brF) / T(\brF)_1$, where $T(\brF)_1$ is the unique parahoric subgroup of $T(\brF)$. Since $T$ is preserved by $\s$, $\s$ acts on $N(\brF)$, and we obtain $\s$-actions on $W_0$ and $\tW$. The natural projection $\tW\rightarrow W_0$ is $\s$-equivariant.

Set $V=X_*(T)_{\G_0} \otimes_\BZ \BR$. The action of $\s$ on $T_{\brF}$ induces a $\s$-action on $V$ (the ``usual'', linear $\s$-action).

Associated to $G\supseteq S$, there is an apartment $\mathcal A$ in the Bruhat-Tits building of $G$ over $\brF$, an affine space under $V$ on which $\tW$ acts by affine transformations. See~\cite[1.2]{Tits:Corvallis}. We obtain a commutative diagram
\[
\xymatrix{
0 \ar[r] & X_*(T)_{\G_0} \ar[r]\ar[d] & \tW \ar[r]\ar[d] & W_0 \ar[r]\ar[d] & 1\\
0\ar[r] & V\ar[r] & \text{Aff}(\CA) \ar[r] & GL(V)\ar[r] & 1.
}
\]
Here $\text{Aff}(\CA)$ denotes the group of affine transformations of the affine space $\CA$. The Frobenius morphism $\s$ on $G$ induces an action on $\mathcal A$.

We fix a $\s$-invariant alcove $\aa$ of $\mathcal A$ and let $\brI$ be the Iwahori subgroup of $G(\brF)$ corresponding to $\aa$.  We fix a special vertex in the
closure $\bar \aa$ of $\aa$ which is fixed under the Frobenius of the (unique) quasi-split inner form of $G$. Using this vertex as the base point, we identify $\mathcal A$ with the underlying affine space of $V$.

By the identification of $\CA$ with $V$, the action of $\s$ on $\CA$ induces an action on $V$ by an affine transformation on $V$. We obtain a natural map $\tW \rtimes \<\s\> \to \text{Aff}(V)$, where $\text{Aff}(V)=V \rtimes GL(V)$ is the group of affine transformations on $V$. Let $p: \tW \rtimes \<\s\> \to GL(V)$ be the projection map. Then $p(\s)$ is precisely the ``usual action'' mentioned above, i.e., the (linear) action induced from the action of $V$. In the sequel, we will always denote this action by $p(\s)$, and denote by $\s$ the action by an affine transformation.

\subsubsection*{The $L$-action}
The dominant chamber is by definition the unique chamber whose closure contains $\aa$. We denote by $V_+$ the closure of the dominant chamber in $V$. Let $B\supseteq T_{\brF}$ be the corresponding Borel subgroup of $G_{\brF}$.

\begin{definition} ($L$-action)
Let $\s_0 = w\circ p(\s)$, where $w\in W_0$ is the unique element such that $\s_0(V_+)=V_+$. This is the {\it L-action} of $\s$ on $V$.
\end{definition}

A slightly different point of view is the following: Let $\Psi_0(G)$ be \emph{the} based root datum of $G$, i.e., the inverse limit over all pairs $(\breve T, \breve B)$ consisting of a maximal torus $\breve T\subseteq G_{\brF}$ and a Borel subgroup $\breve B\supseteq \breve T$ (over $\brF$), of the based root data attached to the pair $(\breve T, \breve B)$. The transition maps are the identifications of root data between different such pairs, say $(\breve T, \breve B)$, $(\breve T', \breve B')$ which come from a suitable inner automorphism $\a$ of $G_{\brF}$ such that $\breve T'=\a(\breve T)$, $\breve B'=\a(\breve B)$. Then $\s$ acts on $\Psi_0(G)$, and the action, when ``evaluated'' on the pair $(T_{\brF}, B)$, is the $L$-action. See~\cite[1.2]{Borel:Corvallis}, \cite[\S 1]{Kottwitz:cuspidal}.

If $G$ is quasi-split, and hence the special vertex used to identify $\CA$ and $V$ is fixed by $\s$, then $\s$ acts linearly on $V$, i.e., $\s= p(\s)$, and fixes the corresponding Weyl chamber, so $\s_0 = p(\s) = \s$.

Now consider a general $G$ again. The root datum $\Psi_0(G)$ depends only on $G_{\brF}$, and the action of $\s$ remains unchanged when $G$ is replaced by an inner form. This gives a way to interpret the $L$-action in terms of the (unique) quasi-split inner form of $G$.

\smallskip

We fix, once and for all, a $W_0 \rtimes \<\s_0 \>$-invariant inner product on $V$.

\subsubsection*{The extended affine Weyl group}
The Iwahori-Weyl group (sometimes called the extended affine Weyl group) $\tW=N(\brF)/(T(\brF) \cap \brI)$ and the finite Weyl group $W_0=N(\brF)/T(\brF)$ are central objects of this paper. Our choice of a special vertex gives rise to a semi-direct product decomposition
\[
\tW=X_*(T)_{\G_0} \rtimes W_0
\]
Note that the inclusion $W_0 \subseteq \tW$ is not $\s$-equivariant in general.
For any $w \in \tW$, we choose a lifting in $N(\brF)$ and denote it by $\dot w$.

The group $\tW$ acts on the apartment $\mathcal A$. Denote by $\Omega\subseteq\tW$ the stabilizer of $\aa$. Then we have $\tW \cong W_a \rtimes\Omega$, where $W_a$ (the ``affine Weyl group'') is the Iwahori-Weyl group of the simply connected cover of the derived group of $G$. Note that $W_a$ is a Coxeter group and one has the Bruhat order and the length function on $W_a$, corresponding to our choice of base alcove. Although the relative root system of $G$ over $\brF$ need not be reduced, there is a unique reduced root system $\Phi$ having the same hyperplanes as affine root hyperplanes and hence having the same affine Weyl group. In the sequel we always work with this reduced root system $\Phi$, and notions such as fundamental coweights refer to $\Phi$. See~\cite{Tits:Corvallis} Section 1.7, or \cite{GHN} Section 3.1.

We extend the length function and the Bruhat order on $W_a$ to $\tW$ in the usual way: for $w, w'\in W_a$, $\tau,\tau'\in\Omega$ we set $\ell(w\tau)=\ell(w)$, and
\[
w\tau \le w'\tau' \quad \Longleftrightarrow\quad \tau=\tau'\text{\ and } w \le w'.
\]
The Frobenius morphism $\s$ on $G$ induces a length-preserving group automorphism on $\tW$ and on $W_a$, which we still denote by $\s$. We denote by $B(\tW)_\s$ the set of $\s$-conjugacy classes in $\tW$.

For any subset $K$ of $\tSS$, the set of simple affine reflections in $\tW$, we denote by $W_K$ the subgroup of $\tW$ generated by the simple reflections in $K$ and by ${}^K \tW$ the set of minimal length representatives in their cosets in $W_K \backslash \tW$.

\subsubsection*{The Newton vector}
Let us recall the definition of the Newton vector of an element $w\in \tW$: Consider $w\s$ as an element of the semidirect product $\tW\rtimes\<\s\>$, where $\<\s\>$ denotes the finite group generated by $\s$ inside the automorphism group of $\tW$. There exists $n$ such that $(w\s)^n = t^\lambda\in \tW\rtimes\<\s\>$, $\lambda \in X_*(T)_{\G_0}$. The \emph{not necessarily dominant Newton vector} of $w$ is $\nu_w := \l/n\in V$. The \emph{Newton vector} $\overline{\nu}_w\in V_+$ of $w$ is by definition the unique dominant element in the $W_0$-orbit of $\nu_w\in V$. It is easy to see that these elements are independent of the choice of $n$.

We recall the definition of straight elements and straight conjugacy classes \cite{HN}. We regard $\sigma$ as an element in the group $\tW \rtimes \langle \sigma \rangle$. The length function on $\tW$ extends in a natural way to a length function on $\tW \rtimes \langle \sigma \rangle$ by requiring $\ell(\sigma)=0$.

We say that an element $w \in \tW$ is {\it $\sigma$-straight} if $\ell(w \sigma(w) \sigma^2(w) \cdots \sigma^{m-1}(w))=m \ell(w)$ for all $m \in \BN$, i.e., $\ell((w \sigma)^m)=m \ell(w)$ for all $m \in \BN$. Equivalently, $w$ is $\s$-straight if and only if $\< \overline{\nu}_w, 2\rho\> = \ell(w)$, where $\rho$ denotes half the sum of the positive finite roots.

We call a $\sigma$-conjugacy class of $\tW$ straight if it contains a $\sigma$-straight element,  and denote by  $B(\tW)_{\sigma-{\rm str}}$  the set of straight $\sigma$-conjugacy classes of $\tW$.

\subsubsection*{The admissible set}
Let $\{\mu\}$ be a conjugacy class of cocharacters of $G$ over $\overline{F}$. Let $\mu \in X_*(T)$ be a dominant representative of the conjugacy class $\{\mu\}$. Let $\underline \mu$ be the image of $\mu$ in $X_*(T)_{\Gamma_0}$. The {\it $\{\mu\}$-admissible set} (\cite{kottwitz-rapoport}, \cite{rapoport:guide}) is defined by
\begin{equation}
\Adm(\{\mu\})=\{w \in \tW;\ w \le t^{x(\underline \mu)} \text{ for some }x \in W_0\}.
\end{equation}

Let $\brK$ be a standard $\sigma$-invariant parahoric subgroup of $G(\brF)$, i.e., a $\s$-invariant parahoric subgroup that contains $\brI$. We denote by $K \subseteq \tSS$ the corresponding set of simple reflections. Then $\s(K)=K$. We denote by $W_K \subseteq \tW$ the subgroup generated by $K$. We set $\Adm^K(\{\mu\})=W_K \Adm(\{\mu\}) W_K$ and $\Adm(\{\mu\})_K=W_K \backslash \Adm^K(\{\mu\})/W_K$.

\subsection{$\s$-conjugacy classes}

\subsubsection*{The Kottwitz classification}
Let $B(G)$ be the set of $\sigma$-conjugacy classes of $G(\brF)$. We denote by $\kappa$ the Kottwitz map (see \cite[(2.1)]{RV}),
$$ \kappa\colon B(G)\to \pi_1(G)_\Gamma$$
with values in the $\G$-coinvariants of the algebraic fundamental group $\pi_1(G)$ of $G$.
We denote by $\overline{\nu}$ the Newton map (see~\cite{kottwitz-isoI}, or \cite[1.1]{HN2}),
$$ \overline{\nu}\colon B(G)\to V_+^{\s_0}.$$
For $w\in \tW$, all representatives in $N(\brF)$ are $\s$-conjugate and their common image under the Newton map is $\overline{\nu}_w$; cf.~Thm.~\ref{thm:sigma-conj-classes} below. We sometimes write $\overline{\nu}_{[b]}$ or $\overline{\nu}_b$ instead of $\overline{\nu}([b])$.

By \cite{kottwitz-isoII}, the map $(\k, \overline{\nu}): B(G) \to \pi_1(G)_\G \times V_+^{\s_0}$ is injective. In other words, a $\s$-conjugacy class is determined by the two invariants: the images of it under the Kottwitz map and under the Newton map.

\smallskip

We have the dominance order on $V_+$. Here $\nu\leq \nu'$ if $\nu'-\nu$ is a non-negative $\BR$-sum of positive relative coroots. For $[b], [b'] \in B(G)$, we say $[b]\leq [b']$ if and only if $\kappa([b])=\kappa([b'])$ and $\overline{\nu}([b])\leq \overline{\nu}([b'])$.

\subsubsection*{The set $B(G, \{\mu\})$}
We set, cf.~\cite{kottwitz-isoII}, \cite{rapoport:guide}, \cite{RV},
 \begin{equation}
B(G, \{\mu\})=\{ [b]\in B(G)\mid \kappa([b])=\mu^\natural, \overline{\nu}([b])\leq \mu^\diamond \} .
 \end{equation}
Here $\mu^\natural$ denotes the common image of $\mu\in\{\mu\}$ in $\pi_1(G)_\Gamma$, and $\mu^\diamond$ denotes the average of the $\s_0$-orbit of $\underline \mu$.   The set $B(G, \{\mu\})$ inherits a partial order from $B(G)$. It has a unique minimal element $[b_0]$, the $\s$-conjugacy class of $\t_0$, where $\t_0$ is the unique element in $\Omega$ such that $t^{\underline \mu} \in W_a \t_0$. Since the Kottwitz map $\kappa$ is constant on $B(G, \{\mu\})$, we may view it as a subset of $V$ via the Newton map.

Let $\BS_0 = W_0 \cap \tSS$ be the set of simple reflections in $W_0$. For any $i \in \BS_0$, let $\o^\vee_i \in V$ be the corresponding fundamental coweight and $\a_i^\vee \in V$ be the corresponding simple coroot. We denote by $\o_i, \a_i \in V^*$ the corresponding fundamental weight and corresponding simple root, respectively. For each $\s_0$-orbit $\CO$ of $\BS_0$, we set
\[
\o_{\CO}=\sum_{i \in \CO} \o_i.
\]
The following description of $B(G, \{\mu\})$ is obtained in \cite[Theorem 1.1 \& Lemma 2.5]{HN2}.

\begin{theorem}\label{BG-mu}
(1) Let $v \in V$. Then $v \in B(G, \{\mu\})$ if and only if $\s_0(v)=v$ is dominant, and for any $\s_0$-orbit $\CO$ on $\BS_0$ with $\<v, \a_i\> \neq 0$ for each (or equivalently, some) $i \in \CO$, we have that $\<\mu+\s(0)-v, \o_{\CO}\> \in \BZ$ and $\<\mu-v, \o_{\CO}\> \ge 0$.

(2) The set $B(G, \{\mu\})$ contains a unique maximal element.
\end{theorem}

\smallskip

We have the following results by the second-named author.

\begin{theorem}\label{thm:sigma-conj-classes}
(1) (\cite[Theorem 3.3]{He14})
The map $$\Psi: B(\tW)_{\sigma-{\rm str}} \to B(G)$$ induced by the inclusion $N(T)(\brF) \subseteq G(\brF)$ is bijective.

(2) (\cite[Proposition 4.1]{He-KR})
Let $\Adm(\{\mu\})_{\sigma-{\rm str}}$ be the set of $\sigma$-straight elements in the admissible set $\Adm(\{\mu\})$ and $B(\tW, \{\mu\})_{\s-{\rm str}}$ be its image in $B(\tW)_{\s-{\rm str}}$. Then $$\Psi(B(\tW, \{\mu\})_{\s-{\rm str}})=B(G, \{\mu\}).$$
\end{theorem}

\begin{theorem}\label{s-newton} (\cite[Theorem 3.5, Theorem 3.7]{He14}, \cite{HR})
Let $w \in \tW$ be a $\s$-straight element. Then all elements of $\brI\dot w\brI$ are $\s$-conjugate to $\dot w$.
\end{theorem}

\subsection{Affine Deligne-Lusztig varieties}

We fix a representative $b\in G(\brF)$ of the $\s$-conjugacy class $[b]$. Fix a level structure $\CK\subseteq G(F)$ corresponding to $K\subseteq \tSS$ as above. Note that $\CK$ is a parahoric subgroup of $G(F)$.

For $w\in \tW_K\backslash \tW /\tW_K$, denote by
\[
X_w(b)_K = \{ g\in G(\brF); g\i b\s(g) \in \brK w\brK \}/\brK
\]
the affine Deligne-Lusztig variety attached to $w$ and $b$. We will omit the subscript $K$ if $K=\emptyset$, i.e., $\brK=\brI$ is the standard Iwahori subgroup.

In the equal characteristic setting, this is the set of $\kk$-valued points of a locally closed subscheme of the partial affine flag variety $G(\brF)/\brK$, equipped with the reduced scheme structure. In the mixed characteristic setting we consider $X_w(b)$ as the $\kk$-valued points of a perfect scheme in the sense of Zhu \cite{Zhu} and Bhatt-Scholze \cite{BS}, a locally closed perfect subscheme of the $p$-adic partial flag variety. We usually denote this (perfect) scheme by the same symbol $X_w(b)_K$, but we are mostly working just with the $\kk$-valued points.

In the theory of local Shimura varieties (or of Rapoport-Zink spaces), it is natural to consider the following union of affine Deligne-Lusztig varieties:
\[
X(\mu, b)_K= \{ g\in G(\brF); g\i b\s(g)\in \bigcup_{w\in\Adm^K(\{\mu\})} \brK w\brK  \}/\brK.
\]

Similarly as before, we can view this as the $\kk$-valued points of a (perfect)
scheme embedded in a partial affine flag variety.  As before, we will omit the
subscript $K$ if $\brK=\brI$.

\subsection{Fine affine Deligne-Lusztig varieties}\label{fine}

Following \cite[3.4]{GH}, we introduce the fine affine Deligne-Lusztig varieties inside the partial affine flag variety $G(\brF)/\brK$.

For $w \in {}^K \tW$ and $b \in G(\brF)$, we set $$X_{K, w}(b)=\{g \brK; g \i b \s(g) \in \brK \cdot_\s \brI w \brI\} / \brK,$$ where $\cdot_\s$ means the $\s$-conjugation. It is the image of $X_w(b)$ under the projection map $G(\brF)/\brI \to G(\brF)/\brK$. It is also proved in \cite[3.4]{GH} that $$X(\mu, b)_K=\sqcup_{w \in \EO{K}} X_{K, w}(b).$$ This decomposition is finer than the decomposition $X(\mu, b)_K=\sqcup_{w \in \Adm(\{\mu\})_K} X_w(b)_K$ and it plays an important role in the study of Shimura varieties with arbitrary parahoric level structure.

This decomposition is analogous, in terms of Shimura varieties, to the decomposition of a Newton stratum into its intersections with the EKOR strata, cf.~\ref{shimura-setup}.

\section{Main results}

\subsection{} We first introduce the notions of Hodge-Newton decomposable pairs $(\mu, b)$, and of minute coweights, both of which play important roles in the main theorem.

\subsubsection{}

For an element $b\in G(\brF)$ with Newton vector $\overline{\nu}_b$, we denote by $M_{\overline{\nu}_b}$ the centralizer of $\overline{\nu}_b$ (cf.~Section~\ref{subsec:defn_Mv}).

\begin{definition}
\begin{enumerate}
\item
Let $M$ be a $\s_0$-stable proper standard Levi subgroup of $G_{\brF}$. We say that $[b] \in B(G, \{\mu\})$ is {\it Hodge-Newton decomposable} with respect to $M$ if $M_{\overline{\nu}_b} \subseteq M$ and $\mu^\diamond-\overline{\nu}_b \in \BR_{\ge 0} \Phi_M^{\vee, +}$. Here $\mu^\diamond=\frac{1}{n_0}\sum_{i=0}^{n_0-1} \s_0^i(\underline{\mu})$ with $n_0 \in \BN$ the order of $\s_0$.
\item
Recall that a pair $(G, \{\mu\})$ is {\it fully Hodge-Newton decomposable} if every non-basic $\s$-conjugacy class $[b]$ is Hodge-Newton decomposable with respect to some proper standard Levi.
\end{enumerate}\label{def:HN-decomp}
\end{definition}

\subsubsection{}

\begin{definition}\label{def:minute}
We say that $\mu$ is {\it minute} (for $G$),  if for any $\s_0$-orbit $\CO$ of $\BS_0$, we have
\[
\<\mu, \o_{\CO}\>+\{\<\s(0), \o_{\CO}\>\} \le 1.
\]
Here $\{\, \}$ is the fractional part of the real number.
\end{definition}

Note that if $G$ is quasi-split, then $\s(0)=0$ and the condition above is that $\<\mu, \o_{\CO}\> \le 1$. However, for non-quasi-split groups, the nontrivial contribution from $\s$ has to be taken into account.

\subsection{Precise version of the main theorem}
\label{subsec:main-thm}

\begin{theorem}\label{main}
Let $(G, \{\mu\}, \brK)$ be as in~\ref{sec:group-theoretic-setup}. Assume that $G$ is quasi-simple over $F$. Then the following conditions are equivalent:
\begin{enumerate}[label=(\arabic*)]
\item (HN-decomp)
The pair $(G, \{\mu\})$ is fully Hodge-Newton decomposable (Def.~\ref{def:HN-decomp});
\item (minute)
The coweight $\mu$ is minute (Def.~\ref{def:minute});
\item (non-basic -- leaves)
For any non-basic $[b]$ in $B(G, \{\mu\})$, $\dim X(\mu, b)_K=0$;
\item (EKOR $\subseteq$ Newton)
For any $w \in \Adm(\{\mu\}) \cap {}^K \tW$, there exists a unique $\s$-conjugacy class $[b] \in B(G, \{\mu\})$ such that $X_w(b) \neq \emptyset$;
\item (basic -- EKOR)
For every $w \in \Adm(\{\mu\}) \cap {}^K \tW$ with $X_w(b_0)\ne \emptyset$, $w\circ \s$ has a fixed point in $\overline{\mathbf a}$.
\item (FC)
For every $w \in \Adm(\{\mu\}) \cap {}^K \tW$ with central Newton vector, $w\circ \s$ has a fixed point in $\overline{\mathbf a}$.
\end{enumerate}
In particular, the validity of the conditions (3)-(6) only depends on the choice of $(G, \{\mu\})$ and is independent of the choice of the parahoric subgroup $\CK$.
\end{theorem}

See Section~\ref{subsec:4-implies-5} for an explanation why condition (5) implies condition (5) in the informal version of the theorem stated in the introduction, namely that $X(\mu, b)_K$ is a union of classical Deligne-Lusztig varieties in a natural way.

\subsection{Passage to adjoint groups}
\label{subsec:passage-to-adjoint}

Let $G$ be a connected reductive group over $F$, and let $\Gad$ be its adjoint group. We denote the automorphism of $\Gad$ induced by $\s$ again by $\s$. Then $\Gad$ decomposes as a product
\[
\Gad \cong G_1 \times \cdots \times G_r,
\]
where each $G_i$ is adjoint and simple over $F$. We can identify
\[
B(G, \{\mu\}) = \prod_i B(G_i, \{\mu_i\}), \quad
\Adm(\{ \mu\}) = \prod_i \Adm(\{ \mu_i\}), \quad
\]
and
\[
\EO{K} = \prod_i \Adm(\{\mu_i\})\cap {}^{K_i}\tW_{G_i}.
\]

It is easy to see that
\[
\dim X(\mu, b)_K = \dim X^{\Gad}(\mu, b)_K = \sum_i \dim X^{G_i}(\mu_i, b_i)_{K_i}.
\]

For all conditions except for condition (3) it is easy to check that they can be verified on the individual factors of this decomposition:

\begin{proposition}\label{adj}
Decompose the adjoint group of $G$ into factors $G_i$ as above. Let (C) be one of the conditions (1), (2), (4), (5), (6) of Theorem~\ref{main}. Then the following are equivalent:
\begin{enumerate}
\item
(C) holds for $(G, \mu, K)$,
\item
(C) holds for all $(G_i, \mu_i, K_i)$.
\end{enumerate}
\end{proposition}

For condition (3), the situation is somewhat different. Consider $[b] \in B(G, \{\mu\})$ that corresponds to $([b_i]) \in \prod B(G_i, \{\mu_i\})$. Then $[b]$ is basic if and only if all the $[b_i]$ are basic. However, if only some, but not all the $[b_i]$ are basic, $[b]$ is still non basic and $X(\mu, b)_K$ could be positive dimensional.

\subsection{Passage to simple groups over $\brF$}\label{passage-brF-simple} \label{simple}
Suppose that $G$ is adjoint and simple over $F$, i.e., the action of $\s$ on the set of connected components of the associated affine Dynkin diagram is transitive. Then we have $$G_{\brF}=G_1 \times \cdots \times G_k,$$ where $G_i$ are isomorphic simple reductive groups over $\brF$, and $\s(G_i)=G_{i+1}$ for all $i$. Here we set $G_{k+1}=G_1$.

Now we construct a simple reductive group $G'$ over $\brF$ as follows. Let $G'_{\brF}=G_k$ and the Frobenius morphism $\s'$ on $G'$ is defined to be $\s^k \mid_{G_k}$.

Since $G$ is adjoint, we have $\s=\t \s_0$, where $\t \in \Omega$ and $\s_0(\BS_0)=\BS_0$. Note that $\s_0$ preserves dominant cocharacters. We set $\mu'=\sum_{j=0}^{k-1} \s_0^j(\mu_{k-j})$, where $\mu=(\mu_1, \cdots, \mu_k)$.

We say that $(G', \{\mu'\})$ is the $\brF$-simple pair associated to $(G, \{\mu\})$. Then by definition, $\mu$ is minute (for $G$) if and only if $\mu'$ is minute for $G'$.

We consider the map
\[
G(\brF) \to G'(\brF), \quad g=(g_1, \cdots, g_k) \mapsto g':=g_k \s(g_{k-1}) \cdots \s^{k-1}(g_1).
\]
This map induces a map from the set $B(G)$ of $\s$-conjugacy classes on $G(\brF)$ to the set $B(G')$ of the $\s'$-conjugacy classes on $G'(\brF)$. By \cite[\S4]{HN2}, this map induces a bijection from $B(G, \{\mu\})$ to $B(G', \{\mu'\})$.

\subsection{Classification}\label{subsec:classification}
We will first describe the classification in terms of the triple $(W_a, \mu, \s)$, where $W_a$ is an affine Weyl group, $\mu$ is a dominant coweight in the associated root system, and $\s$ is a length-preserving automorphism on $W_a$. (Note that in general the translation $t^\mu\in\tW$ will not be an element of $W_a$.)

We use the same labeling of the Coxeter graph as in \cite{Bour}. If $\o^\vee_i$ is minuscule, we denote the corresponding element in $\Omega$ by $\t_i$; conjugation by $\t_i$ is a diagram automorphism of $\tW$ which we denote by $\t_i$ again. Let $\varsigma_0$ be the unique nontrivial diagram automorphism for the finite Dynkin diagram if $W_0$ is of type $A_n, D_n$ (with $n \ge 5$) or $E_6$. For type $D_4$, we also denote by $\varsigma_0$ the diagram automorphism which interchanges $\a_3$ and $\a_4$.

\begin{theorem}\label{triple}
Assume that $G$ is quasi-simple over $\brF$ and $\mu \neq 0$. Then $(G, \{\mu\})$ is fully Hodge-Newton decomposable if and only if the associated triple $(W_a, \mu, \s)$ is one of the following:

\renewcommand{\arraystretch}{1.7}
\[\begin{tabular}{|c|c|c|}
\hline
$(\tilde A_{n-1}, \o^\vee_1, \id)$ & $(\tilde A_{n-1}, \o^\vee_1, \t_1^{n-1})$ & $(\tilde A_{n-1}, \o^\vee_1, \varsigma_0)$
\\
\hline
$(\tilde A_{2m-1}, \o^\vee_1, \t_1 \varsigma_0)$ & $(\tilde A_{n-1}, \o^\vee_1+\o^\vee_{n-1}, \id)$ & $(\tilde A_3, \o^\vee_2, \id)$
\\
\hline
$(\tilde A_3, \o^\vee_2, \varsigma_0)$ & $(\tilde A_3, \o^\vee_2, \t_2)$ &
\\
\hline
\hline
$(\tilde B_n, \o^\vee_1, \id)$ & $(\tilde B_n, \o^\vee_1, \t_1)$ &
\\
\hline\hline
$(\tilde C_n, \o^\vee_1, \id)$ & $(\tilde C_2, \o^\vee_2, \id)$ & $(\tilde C_2, \o^\vee_2, \t_2)$
\\
\hline\hline
$(\tilde D_n, \o^\vee_1, \id)$ &
$(\tilde D_n, \o^\vee_1, \varsigma_0)$ & % $(\tilde D_n, \o^\vee_1, \t_1 \varsigma_0)$
\\
\hline
\end{tabular}
\]
\end{theorem}

The theorem is proved in Section 4.

\subsection{Interpretation}\label{subsec:interpretation}
Let $G$ be adjoint and simple over $F$ and let $\mu$ be a dominant coweight of $G$.

\textbf{Absolutely simple cases.}

The following table lists all cases where $G$ is absolutely simple, $\mu\ne 0$, and $(G, \{\mu\})$ is fully Hodge-Newton decomposable. We use the notation of Tits's table.

\begin{center}
\renewcommand{\arraystretch}{1.7}
\begin{tabular}{|c|c|c|c|}
\hline
$(A_n, \o^\vee_1)$ & $({}^d A_{d-1}, \o^\vee_1)$ & $({}^2 A'_n, \o^\vee_1)$ & $({}^2 A''_{2m-1}, \o^\vee_1)$
\\
\hline
$(A_{n-1}, \o^\vee_1 + \o^\vee_{n-1})$ & $(A_3, \o^\vee_2)$ & $({}^2 A_3, \o^\vee_2)$ & $({}^2 A'_3, \o^\vee_2)$
\\
\hline\hline
$(B_n, \o^\vee_1)$ & $(B$-$C_n, \o^\vee_1)$ & $({}^2 B_n, \o^\vee_1)$ & $({}^2 B$-$C_n, \o^\vee_1)$
\\
\hline\hline
$(C_n, \o^\vee_1)$ & $(C$-$B_n, \o^\vee_1)$ & $(C$-$BC_n, \o^\vee_1)$ & $(C_2, \o^\vee_2)$
\\
\hline
$(C$-$B_2, \o^\vee_2)$ & $(C$-$BC_2, \o^\vee_2)$ & $({}^2 C_2, \o^\vee_2)$ & $({}^2 C$-$B_2, \o^\vee_2)$
\\
\hline\hline
$(D_n, \o^\vee_1)$ & $({}^2 D_n, \o^\vee_1)$ & $({}^2 D'_n, \o^\vee_1)$ &
\\
\hline
\end{tabular}
\end{center}
\medskip

\textbf{Cases where $G$ is not absolutely simple.}

If we drop the assumption the $G$ be absolutely simple, but require that $\mu$ is nontrivial on each $\overline{F}$-factor of $G$, then the pair $(G, \{\mu\})$ is fully Hodge-Newton decomposable if and only if it is in the list of the previous paragraph, or it is one of the following

\begin{itemize}
\item
$(\tilde A_{n-1} \times \tilde A_{n-1}, (\o^\vee_1, \o^\vee_{n-1}), {}^1 \varsigma_0)$.
By ${}^1 \varsigma_0$ we denote the automorphism of $\tilde A_{n-1}\times \tilde A_{n-1}$ exchanging the two factors. (In this case, the passage to a $\brF$-simple group yields the case $(\tilde A_{n-1}, \o^\vee_1+\o^\vee_{n-1}, \id)$.)
\item
$\Res_{E/F} PGL_n$, $\mu$ is $(\o^\vee_1, \o^\vee_{n-1})$ and $E/F$ is a quadratic totally ramified extension. (In this case, the extended affine Weyl group is of type $\tilde{A}_{n-1}$, and the image of $\mu$ in $X_*(T)_{\G_0}$ is $\o^\vee_1+\o^\vee_{n-1}$, so that we again end up with $(\tilde A_{n-1}, \o^\vee_1+\o^\vee_{n-1}, \id)$.)
\end{itemize}

Note that the difference of the numbers of cases in the list in Theorem~\ref{triple} and the lists given here arises as

\begin{itemize}
\item For the triple $(W_a, \mu, \s)$, we consider the (unoriented) affine Dynkin diagram, while for the pair $(G, \{\mu\})$ we consider the (oriented) local Dynkin diagram.
\item Also for the triple $(W_a, \mu, \s)$, we only consider the connected affine Dynkin diagram, while for the pairs $(G, \{\mu\})$ we consider the cases where $\s$ acts transitively on the set of connected components of affine Dynkin diagrams. This leads to the Hilbert-Blumenthal case which corresponds to the triple $(\tilde A_{n-1}, \o^\vee_1+\o^\vee_{n-1}, \id)$.
\end{itemize}

\subsection{}

We list the cases where the description of the basic locus in a Shimura variety has been proved in previous work.

\subsubsection*{Unitary groups}

\begin{enumerate}
\item
The \emph{fake unitary case}.
The $(A_n, \o^\vee_1)$ case arises from $U(1, n)$, $p$ split, and the basic locus is $0$-dimensional. This is the situation considered by Harris and Taylor in \cite{HT}.
\item
The \emph{Drinfeld case}.
The case $({}^d A_{d-1}, \o^\vee_1)$ is the Drinfeld case and the basic locus is the whole Shimura variety.
\item
$U(1, n)$ of a hermitian form:
\begin{itemize}
\item
The case $({}^2 A'_n, \o^\vee_1)$ is the $U(1, n)$, $p$ inert case. For hyperspecial level structure, the basic locus is studied by Vollaard and Wedhorn in \cite{Vollaard-Wedhorn}.
Compare also the article~\cite{Helm-Tian-Xiao} by Helm, Tian and Xiao about products of the form $U(n,1)\times U(1,n)$.
\item
The case $({}^2 A''_{2m-1}, \o^\vee_1)$ is the unramified $U(1, n)$ with nonsplit hermitian form case.
\item
The cases $(B$-$C_n, \o^\vee_1, \tSS - \{n\})$, $({}^2 B$-$C_n, \o^\vee_1, \tSS - \{n\})$ and $(C$-$BC_n, \o^\vee_1, \tSS - \{n\})$ correspond to $U(1, *)$, $p$ ramified. For certain maximal parahoric level structure the basic loci were studied by Rapoport, Terstiege and Wilson in \cite{Rapoport-Terstiege-Wilson}.
For level structure giving rise to exotic good reduction they were studied by Wu \cite{Wu}.
\end{itemize}
\item
The \emph{Hilbert-Blumenthal case}.
\begin{itemize}
\item
The unramified case
$(\tilde A_{n-1} \times \tilde A_{n-1}, (\o^\vee_1, \o^\vee_{n-1}), {}^1 \varsigma_0)$. Note that this case also arises as the local situation attached to certain unitary Shimura varieties, and the corresponding stratification of the supersingular locus is exploited by Tian and Xiao~\cite{Tian-Xiao}, \cite{Tian-Xiao-2} (in the case of hyperspecial level structure). They show that the irreducible components of the supersingular locus provide sufficiently many cycles to ensure that the Tate conjecture holds for the special fiber of the Shimura variety.
The case of Hilbert fourfolds was earlier done by Yu~\cite{Yu}.
\item
The ramified case $\Res_{E/F} PGL_n$, $\mu$ is
$(\o^\vee_1, \o^\vee_{n-1})$ and $E/F$ is a quadratic ramified extension.
\end{itemize}
\end{enumerate}

\subsubsection*{Orthogonal groups}

$SO(2, n)$ of a quadratic form

The cases $(B_n, \o^\vee_1)$, $({}^2 B_n, \o^\vee_1)$, $(C$-$B_n, \o^\vee_1)$, $(D_n, \o^\vee_1)$, $({}^2 D_n, \o^\vee_1)$ and $({}^2 D'_n, \o^\vee_1)$ are the $SO(2, *)$ cases of quadratic forms.

See the paper~\cite{howard-pappas2} by Howard and Pappas for related results, once again for the case of hyperspecial level structure.

\subsubsection*{Exceptional cases}

Some of the exceptional cases have been studied in the case of hyperspecial level structure, for example:

The case $({}^2 A'_3, \o^\vee_2)$ is the $U(2, 2)$, $p$ inert case and for hyperspecial level structure, the basic locus is studied by Howard and Pappas in \cite{howard-pappas}.

The case $(G, \mu, K)=(C_2, \omega_2^\vee)$ is the Siegel case for $Sp(4)$ and for hyperspecial level structure, the basic locus was studied by Katsura and Oort \cite{katsura-oort} and Kaiser \cite{Kaiser}; see also the paper \cite{kudla-rapoport} by Kudla and Rapoport (where the results are applied to computing intersection numbers of arithmetic cycles).

For the other exceptional cases, we do not know of a reference in the literature, but there is no question that they can be dealt with in a similar way.

\subsubsection*{Cases which do not arise from a Shimura variety}
Not all the cases come from a Shimura variety: $(\tilde C_n, \o^\vee_1, \id)$ does not, and $(\tilde A_{n-1}, \o_1^\vee+\o_{n-1}^\vee, \id)$ does not directly.

\section{Hodge-Newton decomposition}
\label{sec:HN}

\subsection{}\label{subsec:defn_Mv}
In this section, we assume that the group $G$ is adjoint.

Recall that we fix a $W_0 \rtimes \<\s_0\>$-invariant positive definite symmetric bilinear form $\<\, , \>$ on $V$. Let $\Phi$ be the set of roots in the underlying reduced root system of the relative root system of $G$ (over $\brF$). For any $v \in V$, we set $\Phi_{v, 0}=\{a \in \Phi; \<a, v\>=0\}$ and $\Phi_{v, +}=\{a \in \Phi; \<a, v\>>0\}$. Let $M_v \subseteq G(\brF)$ be the Levi subgroup generated by $T$ and $U_a$ for $a \in \Phi_{v, 0}$ and let $N_v \subseteq G(\brF)$ be the unipotent subgroup generated by $U_a$ for $a \in \Phi_{v, +}$. Put $P_v=M_v N_v$, which is a semistandard parabolic subgroup of $G$ with $N_v$ being its unipotent radical. Then $M_v$ and $P_v$ are a semistandard Levi subgroup and a parabolic subgroup respectively. Moreover, $P_v$ is standard (with respect to the fixed Borel $B \subseteq G_{\brF}$ over $\brF$) if and only if $v$ is dominant.

If $v \in X_*(T)_{\G_0}\otimes \BQ \cong X_*(T)^{\G_0}\otimes \BQ$, we can view $v$ as a homomorphism $\BD \rightarrow T_{\brF} \rightarrow G_{\brF}$. Here $\BD$ denotes the pro-torus with character group $\BQ$.  Then $M_v$ is the centralizer of this homomorphism inside $G_{\brF}$. In this sense, our notation is compatible with Kottwitz's notation in~\cite{kottwitz-isoI}.

Let $\tW_v=\tW(M_v)$ be the Iwahori-Weyl group of $M_v$ and $\le_v$ be the Bruhat order on $\tW_v$. It can be described in the following way: for $x, y \in \tW_v$, $x \le_v y$ if and only if there exist affine reflections $r_1, \cdots, r_n$ of $\tW_v$ such that $x \le r_1 x \le \cdots \le r_n \cdots r_1 x=y$. Here $\le$ is the usual Bruhat order on $\tW$.

As the dominant $M_v$-chamber we take the unique $M_v$-chamber which contains the fixed dominant chamber for $G_{\brF}$. Below we say that $v' \in V$ is $P_v$-dominant, if it is $M_v$-dominant and $\a(v') \ge 0$ for all $\a\in \Phi_{v, +}$, i.e., for all $\a$ occurring in $N_v$. This yields one chamber inside $V$, and in particular for every $v' \in V$, the $W_0$-orbit of $v'$ contains a unique $P_v$-dominant element.

By a superscript, we denote conjugation, e.g., ${}^z M = zMz\i$.

\subsection{} We recall the definition of alcove elements introduced in \cite{GHKR} and \cite{GHN} and discuss some properties of them.

\begin{definition}
Let $w \in \tW$ and $v \in V$. We say $w$ is a {\it $(v, \sigma)$-alcove element} if
\begin{itemize}
\item
$p(w \s)(v)=v$;
\item
$N_v(\brF) \cap \dot w \brI \dot w \i \subseteq N_v(\brF) \cap \brI$.
\end{itemize}
\end{definition}

The following result follows directly from the definition.

\begin{lemma}\label{alcove-1}
Let $v, v' \in V$. If $\Phi_{v, 0} \subseteq \Phi_{v', 0}$ and $\Phi_{v', +} \subseteq \Phi_{v, +}$, then any $(v, \s)$-alcove element is automatically a $(v', \s)$-alcove element.
\end{lemma}

\begin{lemma}\label{nu-v}
Let $x \in \tW$ and $v \in V$ such that $x$ is a $(v, \s)$-alcove element. Then  $$\<\nu_x, v\>=\<x \s(e)-e, v\>$$ for any $e \in V$.
\end{lemma}

\begin{proof}
Assume that $x \s=t^\l z$ for $\l \in X_*(T)_{\G_0}$ and $z \in W_0 \rtimes \<\s_0\>$. Since $x$ is a $(v, \s)$-alcove element, $z(v)=v$. Then $$\<x \s(e)-e, v\>=\<z(e)-e, v\>+\<\l, v\>=\<\l, v\>.$$ We also have $\<z^i(\l), v\>=\<\l, v\>$ for any $i \in \BZ$. Thus $\<\nu_x, v\>=\<\l, v\>$. The lemma is proved.
\end{proof}

\begin{lemma}\label{le-v}
If $v \in V$ and $w$ is a $(v, \s)$-alcove element, then $w'$ is a $(v, \s)$-alcove element for any $w' \le_v w$.
\end{lemma}

\begin{proof}
We can assume $w'=s w$ for some reflection $s \in \tW_v$. Suppose that there exists $k \in \BZ$ such that $\a(q) > k > \a(sw(q))$ for $\a \in \Phi_{v, +}$ and $q \in \aa$. Since $w$ is a $(v, \s)$-alcove, we have $$\a(w(q)), \a(q)>k>\a(s w(q)).$$ Since $s \in \tW_v$, $p(s)=s_\g$ for some $\g \in \Phi_{v, 0}$. By replacing $\g$ with $-\g$ if necessary, we can assume $w(q)-sw(q) \in \mathbb{R}_{>0} \g^\vee$. Hence $\a(\g^\vee)>0$. Since $s w < w$, $w \aa$ and $s \aa$ are on the same side of the hyperplane $V^s$, and $\aa$ and $s w \aa$ are on the other side of the hyperplane $V^s$. Since $w(q)-sw(q) \in \mathbb{R}_{>0} \g^\vee$, we have $s(q)-q \in \mathbb{R}_{>0} \g^\vee$ and $\a(s(q)) > \a(q) > k > \a(s w(q))$, which means $s \aa >_\a sw \aa$ (with notation as in~\cite[2.1]{GHKR}). Now we have $\aa=s s \aa >_{p(s)(\a)} s sw \aa= w \aa$, contradicting the assumption that $w$ is a $(v, \s)$-alcove element. Therefore $s w \aa \ge_\a \aa$ for all $\a \in \Phi_{v, +}$. Thus $w'$ is a $(v, \s)$-alcove element.
\end{proof}

The following result is proved in \cite[Lemma 2.1]{Ni}.
\begin{lemma}\label{alcove-2}
Let $v \in V$ and $w \in \tW$ with $p(w\s)(v)=v$. Let $s$ be a reflection in $\tW_v$ with $\ell(s w)=\ell(w)+1$. Then $w$ is a $(v, \s)$-alcove element if and only if $s w$ is a $(v, \s)$-alcove element.
\end{lemma}

\subsection{}
For any semistandard $\s$-stable Levi subgroup $M$, we set $\brI_M=\brI \cap M(\brF)$, an Iwahori subgroup of $M$, and $\mathcal{FL}(M)=M(\brF)/\brI_M$. For any $b \in M(\brF)$ and $w \in \tW(M) \subseteq \tW$, we define $$X^M_w(b)=\{m \brI_M; m \i b \s(m) \in \brI_M \dot w \brI_M\}.$$ Then we have a natural embedding $$X^M_w(b) \to X^G_w(b).$$

Recall that the notation $X_w(b)$ and $X(\mu, b)$ without subscript ${-}_K$ always refers to the Iwahori case $K=\emptyset$.

The following result is proved in \cite[Theorem 2.1.4]{GHKR} for split groups\footnote{But note that the statement of \cite[Theorem 2.1.4]{GHKR} is not entirely correct. As the proof given there shows, we need to assume that $X^M_x(b)\ne \emptyset$. In the statement below this is reflected by assuming that $\bar \nu_{b}^{M_v}$ is $P_v$-dominant. In particular, the statement right after \cite[Theorem 2.1.4]{GHKR}, claiming that \emph{for $x\mathfrak a$ a $P$-alcove, $X^M_x(b)\ne \emptyset$ if and only if $X^G_x(b)\ne \emptyset$}, is false. This can be seen already for $G=SL_2$, $M$ the diagonal torus, $x$ a regular dominant translation element and $b$ the anti-dominant translation obtained by conjugating $x$ by the non-trivial element of the finite Weyl group.} and in \cite[Proposition 2.5.1 \& Theorem 3.3.1]{GHN} in general. As usual, we denote by $\JJ_b$ the $\s$-centralizer of $b$:
\[
\JJ^G_b = \JJ_b = \{ g\in G(\brF);\ g\i b \s(g) = b \},
\]
and likewise we write $\JJ^{M_v}_b$ for the $\s$-centralizer inside $M_v(\brF)$.

\begin{theorem}\label{alcove-adlv}
Let $v \in V^{p(\s)}$ and let $b \in M_v(\brF)$ such that $\bar \nu_{b}^{M_v}$ is $P_v$-dominant. If $w$ is a $(v, \s)$-alcove element, then the closed immersion $\mathcal{FL}(M_v) \to \mathcal{FL}(G)$ induces a bijection $$\JJ^{M_v}_{b} \backslash X^{M_v}_w(b) \xrightarrow{\sim} \JJ^G_{b'} \backslash X^G_w(b).$$ In particular, if $\JJ^G_{b} \subseteq M_v(\brF)$, then we have a natural isomorphism $X^{M_v}_w(b) \xrightarrow{\sim} X^G_w(b)$. Here $P_v$ denotes the standard parabolic subgroup associated to $v$.
\end{theorem}

\subsection{}\label{HN-setup}
In the rest of this section, we fix a dominant element $v^\flat$ of $V$ with $\s_0(v^\flat)=v^\flat$ and let $P=P_{v^\flat}$ (resp. $M=M_{v^\flat}$) be the standard parabolic subgroup associated to $v^\flat$. Set $J=J_{v^\flat} = \{ s\in \BS_0;\ s(v^\flat)=v^\flat \}$ and denote by $W_0^J \subseteq W_0$ the set of minimal length representatives in their cosets in $W_0/W_J$. We assume that $(\mu, b)$ is Hodge-Newton decomposable with respect to $M$, which means by Definition~\ref{def:HN-decomp} that $M_{\overline{\nu}_b}\subseteq M$ and $\mu^\diamond-\overline{\nu}_b\in\BR_{\ge 0}\Phi^{\vee, +}_M$. In particular, $\<\mu, v^\flat\>=\<\mu^\diamond, v^\flat\>=\<\overline{\nu}_b, v^\flat\>$.

Let $\fkP^\s$ be the set of $\s$-stable parabolic subgroups of $G$ containing the maximal torus $T$ which are conjugate to $P$. For $P' \in \fkP^\s$, we have $P'=M' N'$, where $N'$ is the unipotent radical of $P'$ and $M' \subseteq P'$ is the unique Levi subgroup containing $T$. Since $P'$ is conjugate to $P$, there exists a unique element $z_{P'} \in W_0^J$ such that ${}^{\dot z_{P'}} P=P'$.

\begin{lemma} \label{para}
Let $P' \in \fkP^\s$. Then we have $p(\s) (z_{P'}(v^\flat))=z_{P'}(v^\flat)$
\end{lemma}
\begin{proof}
Note that $P'=P_{z_{P'}(v^\flat)}$. Since $\s(P')=P'$, we have  $P_{z_{P'}\i p(\s) (v^\flat)}=P_{v^\flat}$. Since $p(\s)=w \s_0$ for some $w \in W_0$ and $\s_0(v^\flat)=v^\flat$, we deduce that $P_{z_{P'}\i w \s_0(z_{P'})(v^\flat)}= P_{v^\flat}$. Notice that two $W_0$-conjugate vectors in $V$ are the same if they have the same associated parabolic subgroup. Thus we have $z_{P'}\i w \s_0(z_{P'})(v^\flat)=v^\flat$ as desired.
\end{proof}

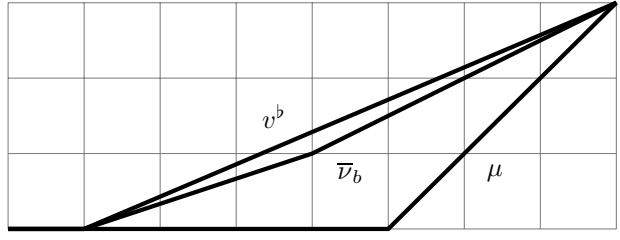
\begin{figure}
\begin{tikzpicture}
\tkzInit[xmax=8,ymax=3,xmin=0,ymin=0]
\tkzGrid
\draw[ultra thick] (0,0) -- (5,0) -- (8,3);
\draw[ultra thick] (0,0) -- (1,0) -- (4,1) -- (8,3);
\draw[ultra thick] (0,0) -- (1,0) -- (8,3);

\node at (6.4, .75) {$\mu$};
\node at (4.5, .75) {$\overline{\nu}_b$};
\node at (3.5, 1.5) {$v^\flat$};

\end{tikzpicture}
\caption{An example for $G = PGL_8$, $\mu = (1,1,1,0,0,0,0,0)$, $\overline{\nu}_b = (\frac 12, \frac 12, \frac 12, \frac 12, \frac 13, \frac 13, \frac 13, 0)$. Then $b$ is HN-decomposable with respect to $M_{v^\flat}$ with $v^\flat = (\frac 37, \frac 37, \frac 37, \frac 37, \frac 37, \frac 37, \frac 37, 0)$ (but not with respect to $M_{\overline{\nu}_b}$). The picture shows the ``Newton polygons'' of $v^\flat$, $\overline{\nu}_b$ and $\mu$, following the usual convention that the slopes of the polygon segments are the entries of the anti-dominant representative.
In terms of the picture, the property that $(\mu, b)$ is HN decomposable corresponds to the fact that the Newton polygon of $\overline{\nu}_b$ has a break point which lies on the Newton polygon of $\mu$.}
\end{figure}

The main purpose of this section is to establish the Hodge-Newton decomposition for $X(\mu, b)_K$, i.e., to show that $X(\mu, b)_K$ is the disjoint union of $X^{M'}(*, *)_*$, where $M'$ are certain semistandard Levi subgroups that are conjugate to the given Levi $M$. In order to do this, we first show that for each such $M'$, there exists a representative of $[b]$ in $M'(\brF)$.

\begin{proposition} \label{choice}
Let $[b] \in B(G)$ and $P'=M' N' \in \fkP^\s$. Then

(1) there exists an element $b_{P'} \in [b] \cap M'(\brF)$, unique up to $M'(\brF)$-$\s$-conjugation, such that $\overline \nu_{b_{P'}}^{M'}$ is $P'$-dominant;

(2) $\JJ_{b_{P'}} \subseteq M'(\brF)$.
\end{proposition}

\begin{proof}
By \cite[\S3]{He14}, the $\s$-conjugacy class $[b]$ contains a representative $\dot w$ for some $w \in \tW$. Then $\nu_w \in W_0 \cdot \overline{\nu}_b$ (cf.~Section~\ref{sec:group-theoretic-setup}). Let $x \in W_0$ with $x (\nu_w)=z_{P'}(\overline{\nu}_b)$. Set $w'=x w \s(x) \i$. Then $\nu_{w'}=z_{P'} (\overline{\nu}_b)$ is $P'$-dominant and $p(w' \s)(z_{P'}(\overline{\nu}_b))=z_{P'}(\overline{\nu}_b)$. By the choice of $v^\flat$, we deduce that $p(w' \s)(z_{P'}(v^\flat))=z_{P'}(v^\flat)$. On the other hand, by Lemma \ref{para}, we have $p(\s)(z_{P'}(v^\flat))=z_{P'}(v^\flat)$. Thus $p(w')(z_{P'}(v^\flat))=z_{P'}(v^\flat)$, that is, $\dot w' \in [b] \cap M'(\brF)$ as desired. The uniqueness of $b_{P'}$ follows from Kottwitz's classification \cite{kottwitz-isoII} of $\s$-conjugacy classes of $M'(\brF)$. So (1) is proved.

Part (2) follows from results of Kottwitz (who also shows that $\JJ_{b_{P'}}$ is an inner form of $M'$), see \cite[Remark 6.5]{kottwitz-isoI}, \cite[\S 3.3, \S 4.3]{kottwitz-isoII}.
\end{proof}

\subsection{}\label{HN-setup-2} Let $\brK \subseteq G(\brF)$ be a standard $\s$-stable parahoric subgroup, corresponding to $K\subseteq \tSS$. Let $\fkP^\s / W_K^\s$ be the set of $W_K^\s$-conjugacy classes of parabolic subgroups in $\fkP^\s$. For each $P'=M' N' \in \fkP^\s$, we choose a representative $b_{P'} \in M'(\brF) \cap [b]$ such that $\bar \nu_{b_{P'}}^{M'}$ is $P'$-dominant. We have an embedding
\begin{equation}\label{map-from-levi}
\phi_{P', K}: X^{M'}(\mu_{P'}, b_{P'})_{K_{M'}} \to X^G(\mu, b_{P'})_{K} \to X^G(\mu, b)_{K},
\end{equation}
where the first map comes from the closed immersion $\mathcal{FL}(M') \to \mathcal{FL}(G)$ and the second map is given by $g \brK \mapsto h_{P'} g \brK$. Here $\brK_{M'}=M'(\brF) \cap \brK \subseteq M'(\brF)$ is a parahoric subgroup and $K_{M'}$ is the corresponding set of simple reflections in the Iwahori-Weyl group $\tW(M')$ of $M'$; $\mu_{P'}=z_{P'}(\mu)$ is the unique $P'$-dominant conjugate of $\mu$ and $h_{P'}$ is an element in $G(\brF)$ with $b_{P'}=h_{P'}\i b \s(h_{P'})$. Thanks to Proposition \ref{choice}, the image of $\phi_{P', K}$ does not depend on the choices of $b_{P'}$ and $h_{P'}$, but only depends on $P'$. We simply write $\phi_{P'}$ for $\phi_{P', \emptyset}$.

\smallskip

We first discuss the $W_K^\s$-action on $\fkP^\s$.
\begin{lemma} \label{f1}
Let $P'=M' N', P''=M'' N'' \in \fkP^\s$. The following conditions are equivalent:

(a) $P', P''$ are conjugate under $W_K$;

(b) $\brK \dot z_{P'} M(\brF)=\brK \dot z_{P''} M(\brF)$;

(c) $P', P''$ are conjugate under $W_K^\s$.

Moreover, in these cases, there exists $u \in W_K^\s$ such that ${}^{\dot u} (P')=P''$, ${}^{\dot u} \brI_{M'}=\brI_{M''}$ and $$u \Adm^{M'}(\{z_{P'}(\mu)\}) u\i = \Adm^{M''}(\{z_{P''}(\mu)\}).$$
\end{lemma}
\begin{proof}
(b) $\Rightarrow$ (a): By definition ${}^{\dot z_{P'}} \brI_M \subseteq \brI$. Then $$\brK \dot z_{P'} M(\brF)=\brK \dot z_{P'} \brI_M \tW(M) \brI_M=\brK \dot z_{P'} \tW(M) \brI_M \subseteq \brI W_K \brI \dot z_{P'} \tW(M) \brI.$$
Note that for any $x, y \in \tW$, $\brI x \brI y \brI \subseteq \cup_{x' \le x} \brI x' y \brI$. We have $$\brK \dot z_{P'} M(\brF) \subseteq \brI W_K \dot z_{P'} \tW(M) \brI.$$ Similarly, $\brK \dot z_{P''} M(\brF) \subseteq \brI W_K \dot z_{P''} \tW(M) \brI$. Thus $W_K z_{P'} \tW(M)=W_K z_{P''} \tW(M)$. In other words, $z_{P''}=u_1 z_{P'} x$ for some $u_1 \in W_K$ and $x \in \tW_M$. Therefore, $P''={}^{\dot z_{P''}} P={}^{\dot u_1 \dot z_{P'} \dot x} P={}^{\dot u_1 \dot z_{P'}} P={}^{\dot u_1} (P')$ and (a) follows.

(c) $\Rightarrow$ (b): Let $w \in W_K^\s$ with ${}^{\dot w} (P')=P''$. It suffices to show \begin{align*}\tag{$\ast$} w z_{P'} \in z_{P''} \tW(M). \end{align*} Indeed, noticing that $P''={}^{\dot w} (P')={}^{\dot w \dot z_{P'}} P=P_{p(w) z_{P'}(v^\flat)}=P_{z_{P''}(v^\flat)}$, we deduce that we have $p(w) z_{P''}(v^\flat)=z_{P''}(v^\flat)$. Hence $p(z_{P''}\i w z_{P'})$ lies in the relative finite Weyl group of $M$ and $z_{P''}\i w z_{P'}$ lies in the Iwahori-Weyl group $\tW(M)$. So ($\ast$) follows as desired.

(a) $\Rightarrow$ (c): Let $\tW(M') \subseteq \tW$ be the Iwahori-Weyl group of $M'$ and let $W_{M', K}=\tW(M') \cap W_K$. Since $W_{M', K} \subseteq W_K$ is a reflection subgroup, each $W_{M', K}$-coset in $W_K$ contains a unique minimal element with respect to the Bruhat order on $W_K$ (see \cite[Corollary 3.4]{Dyer}). We set $W_K^{M'}=\{\min(w W_{M', K}); w \in W_K\}$. Since $\tW(M'), W_K$ are $\s$-stable, $W_K^{M'}$ is also $\s$-stable. By assumption there exists $u \in W_K$ such that ${}^{\dot u} (P')=P''$. Moreover, noticing that $P'$ is normalized by $W_{M', K}$, we can and do assume further that $u \in W_K^{M'}$. Since $P', P''$ are $\s$-stable, we have ${}^{\dot u\i \s(\dot u)} (P')=P'$. This means $u\i \s(u) \in \tW(M')$ and hence $u W_{M', K}=\s(u) W_{M', K}$. Noticing that $u, \s(u) \in W_K^{M'}$, we deduce that $u=\s(u)$ and (c) is proved.

It remains to verify the ``Moreover'' part. Let $u \in W_K^\s \cap W_K^{M'}$ be as in the above paragraph. Then ${}^{\dot u} \brI_{M'} \subseteq M''(\brF) \cap \brI=\brI_{M''}$ and hence ${}^{\dot u} \brI_{M'}=\brI_{M''}$. So the map $w \mapsto u w u\i$ preserves the Bruhat orders on $\tW(M'), \tW(M'')$ associated to the Iwahori subgroups $\brI_{M'}, \brI_{M''}$ respectively. In particular, we have $$u \Adm^{M'}(\{z_{P'}(\mu)\}) u\i =\Adm^{M''}(\{p(u) z_{P'}(\mu)\})=\Adm^{M''}(\{z_{P''}(\mu)\}),$$ where the second equality follows from ($\ast$). The proof is finished.
\end{proof}

\subsection{} Let $P'=M' N' \in \fkP^\s$. Consider the following commutative diagram
\[
\xymatrix{
  X^{M'}(\mu_{P'}, b_{P'}) \ar[d]_-{\pi_{K_{M'}}^{M'}} \ar[r]^-{\phi_{P'}} & X(\mu, b) \ar[d]^{\pi_{K}^G} \\
  X^{M'}(\mu_{P'}, b_{P'})_{K_{M'}} \ar[r]^-{\phi_{P', K}} & X(\mu, b)_{K}, }
\]
where $\pi^G_{K}, \pi^{M'}_{K_{M'}}$ are the natural projections. The following result describes the relation between the images of $\phi_{P'}$ for various $P'$. Here $K_{M'}$ is the subset of simple reflections of $\tW(M')$ corresponding to the standard parahoric subgroup $\brK \cap M'(\brF) \subseteq M'(\brF)$ of $M'(\brF)$.

\begin{proposition} \label{disjoint}
Let $\brK$ be a standard $\s$-stable parahoric subgroup corresponding to $K \subseteq \tilde \BS$. Let $P'=M' N', P''=M'' N'' \in \fkP^\s$. Then $\Im(\pi_K^G \circ \phi_{P'}) \cap \Im(\pi_K^G \circ \phi_{P''}) \neq \emptyset$ only if $P', P''$ are conjugate under $W_K^\s$. Moreover, in the latter case, we have $\Im(\pi_K^G \circ \phi_{P'})=\Im(\pi_K^G \circ \phi_{P''})$.
\end{proposition}
\begin{proof}
Assume $\Im(\pi_K^G \circ \phi_{P'}) \cap \Im(\pi_K^G \circ \phi_{P''}) \neq \emptyset$. Then $$(h_{P'} M'(\brF) \brK) \cap (h_{P''} M''(\brF) \brK) \neq \emptyset.$$ Let $b_1=b_{P''}=(h_{P'}\i h_{P''})\i b_{P'} \s(h_{P'}\i h_{P''})$ and $b_2=(\dot z_{P'} \dot z_{P''}\i)\i b_{P'} \s(\dot z_{P'} \dot z_{P''})$. Then $b_1, b_2 \in M''(\brF) \cap [b]$ (to check that $b_2\in M''(\brF)$, show that $(z_{P'}z_{P''}\i)\i \s(z_{P'}z_{P''}\i) \in W(M'')$, compare the proof of (a)$\Rightarrow$(c) of Lemma~\ref{f1}) and $\bar \nu_{b_1}^{M''}=\bar \nu_{b_2}^{M''}$ is $P''$-dominant. By Proposition \ref{choice}, $b_1, b_2$ are $\s$-conjugate under $M''(\brF)$ and hence $M'(\brF) \dot z_{P'} \dot z_{P''}\i M''(\brF)=M'(\brF) h_{P'}\i h_{P''} M''(\brF)$. Thus we have $$\dot z_{P'}\i M'(\brF) \brK \cap \dot z_{P''}\i M''(\brF) \brK \neq \emptyset.$$ Notice that ${}^{\dot z_{P'}\i} (M')={}^{\dot z_{P''}\i} (M'')=M$. We deduce that $M(\brF) \dot z_{P'}\i \brK=M(\brF) \dot z_{P''}\i \brK$. By Lemma \ref{f1}, $P', P''$ are conjugate under $W_K^\s$.

Assume $P', P''$ are conjugate under $W_K^\s$. Let $u \in W_K^\s$ be as in Lemma \ref{f1}. We may take $h_{P''}=h_{P'} \dot u\i$ and $b_{P''}=\dot u b_{P'} \dot u\i$. Now we have
\begin{align*}
\Im(\pi^G_K \circ \phi_{P'}) &= \{h_{P'} g' \brK; g' \in M'(\brF), (g')\i b_{P'} \s(g') \in \brI_{M'} \Adm^{M'}(\{z_{P'}(\mu)\}) \brI_{M'} \} \\
&=\{h_{P'} \dot u\i \dot u g' \dot u\i \brK;  g' \in M'(\brF), (g')\i b_{P'} \s(g') \in \brI_{M'} \Adm^{M'}(\{z_{P'}(\mu)\}) \brI_{M'} \} \\ &=\{h_{P''} g'' \brK;  g'' \in M''(\brF), (g'')\i b_{P''} \s(g'') \in \dot u \brI_{M'} \Adm^{M'}(\{z_{P'}(\mu)\}) \brI_{M'} \dot u\i \} \\
&=\{h_{P''} g'' \brK;  g'' \in M''(\brF), (g'')\i b_{P''} \s(g'') \in \brI_{M''} u \Adm^{M'}(\{z_{P'}(\mu)\}) u\i \brI_{M''} \} \\
&=\{h_{P''} g'' \brK;  g'' \in M''(\brF), (g'')\i b_{P''} \s(g'') \in \brI_{M''} \Adm^{M''}(\{z_{P''}(\mu)\}) \brI_{M''}\} \\
&=\Im(\pi^G_K \circ \phi_{P''}).
\end{align*}
The proposition is proved.
\end{proof}

\subsection{} We set $$\Adm(\{\mu\}, b)=\{w \in \Adm(\{\mu\}); X_w(b) \neq \emptyset\}.$$ For each $P'=M' N' \in \fkP^\s$, we set $$\Adm^{M'}(\{z_{P'}(\mu)\}, b)=\{w \in \Adm^{M'}(\{z_{P'}(\mu)\}); X_w(b) \neq \emptyset\}.$$ We have the following decomposition of the admissible set, which plays a crucial role in the study of Hodge-Newton decomposition (see Theorem \ref{HN-dec} below).

\begin{theorem}\label{adm-adm}
Under the assumptions in Section~\ref{HN-setup}, we have

(1) $\Adm(\{\mu\}, b)=\bigsqcup_{P'=M' N' \in \fkP^\s} \Adm^{M'}(\{z_{P'}(\mu)\}, b)$.

(2) Let $P'=M' N' \in \fkP^\s$. Then each element of $ \Adm^{M'}(\{z_{P'}(\mu)\}, b)$ is a $(z_{P'}(v^\flat), \s)$-alcove element.
\end{theorem}

Before proving this theorem, we first need several technical Lemmas.

\begin{lemma}\label{x-sx}
Let $x \in \Adm(\{\mu\})$ and $v \in W_0 \cdot v^\flat$ such that $x$ is a $(v, \s)$-alcove element and $\<\nu_x, v\>=\<\mu, v^\flat\>$. Let $s$ be an affine reflection in $\tW$ such that $\ell(s x)=\ell(x)+1$ and $s x \in \Adm(\{\mu\})$, then $s \in \tW_v$. Moreover, $sx$ is a $(v, \s)$-alcove element and $\<\nu_{s x}, v\>=\<\mu, v^\flat\>$.
\end{lemma}

\begin{proof}
Set $y=s x$. Let $\a$ be a finite root with $p(s)=s_\a$. Let $e$ be the barycenter of the fundamental alcove $\aa$. Then $s(e)-e$ and $p(s)(e)-e \in \BR \a^\vee$. Therefore $x e, y e, p(s) y e+e-p(s)(e)$ lie on the same line of $V$.

Since $x < y$, both $e$ and $x e$ are separated from $y \s (e)$ by the affine root hyperplane $H_s$. Then $x e-e$ lies in the open convex hull of $y e-e$ and $p(s)(y e-e)$. Therefore $\<x e-e, v\> \le \max\{\<y e-e, v\>, \<p(s)(y e-e), v\>\}$.

We have
\begin{align*} \<\mu, v^\flat\> &=\<\nu_x, v\>=\<x e-e, v\> \le \max\{\<y e-e, v\>, \<p(s)(y e-e), v\>\} \\ & \le \<\overline{y e-e}, v^\flat\> \le \<\mu, v^\flat\>.\end{align*} Here $\overline{y e-e}$ denotes the unique dominant element in the Weyl group orbit of $ye-e$, the second equality follows from Lemma~\ref{nu-v} and the last inequality follows from  $y \in \Adm(\{\mu\}) \subseteq \text{Perm}(\{\mu\})$, the $\{\mu\}$-permissible set, see~\cite{kottwitz-rapoport}. Therefore, $$\<x e-e, v\>=\max\{\<y e-e, v\>, \<p(s)(y e-e), v\>\}.$$ Since $y e-e, p(s)(y e-e) \in x \s (e)-e+\BR_{\neq 0} \, \a^\vee$, we have $\<\a^\vee, v\>=0$ and $\<y e-e, v\>=\<x e-e, v\>$. By Lemma~\ref{alcove-2}, $y$ is a $(v, \s)$-alcove element. By Lemma~\ref{nu-v}, \[\<\nu_y, v\>=\<y e-e, v\>=\<x e-e, v\>=\<\mu, v^\flat\>.\]
\end{proof}

\begin{corollary}\label{going-up}
Let $x, v, \mu$ be as in Lemma~\ref{x-sx}. Let $w \in \Adm(\{\mu\})$ such that $x \le w$. Then $x \le_v w$, $w$ is a $(v, \s)$-alcove element and $\<\nu_w, v\>=\<\mu, v^\flat\>$.
\end{corollary}
\begin{proof}
Since $x \le w$, there exists a sequence of affine reflections $s_1, s_2, \cdots, s_{\ell(w)-\ell(x)}$ such that $$x<s_1 x<s_2 s_1 x<\cdots<s_{\ell(w)-\ell(x)} \cdots s_1 x=w.$$ By Lemma~\ref{x-sx}, one can prove by induction on $i$ that $s_i \in \tW_v$, $s_i \cdots s_1 x$ is a $(v, \s)$-alcove element and $\<\nu_{s_i \cdots s_1 x}, v\>=\<\mu, v^\flat\>$ for all $i$. In particular, $x \le_v w$, $w$ is a $(v, \s)$-alcove element and $\<\nu_w, v\>=\<\mu, v^\flat\>$.
\end{proof}

\begin{lemma}\label{conn1}
Let $C$ be a connected component of $X(\mu, b)$. Let $x, y \in \Adm(\{\mu\})$ such that $X_x(b) \cap C$ and $X_y(b) \cap C$ are nonempty. If $x$ is a $(v, \s)$-alcove element and $\<\nu_x, v\>=\<\mu, v^\flat\>$ for some $v \in V^{p(\s)} \cap W_0 \cdot v^\flat$ , then $y$ is a $(v, \s)$-alcove element and $\<\nu_y, v\>=\<\mu, v^\flat\>$.
\end{lemma}
\begin{proof}
By assumption (cf.~\cite{EGAInew} Ch.~0, Cor.~2.1.10), there exist irreducible components $Z_1, \dots, Z_m$ of $C$ such that $X_x(b) \cap Z_1 \neq \emptyset$, $X_y(b) \cap Z_m \neq \emptyset$ and $Z_i \cap Z_{i+1} \neq \emptyset$ for $1 \le i \le m-1$. Let $w_i, w_{i, i+1} \in \Adm(\mu)$ such that $X_{w_{i,i+1}}(b) \cap Z_i \cap Z_{i+1} \ne \emptyset$ and $X_{w_i}(b) \cap Z_i$ is dense in $Z_i$. In particular, we have $x \le w_1$, $y \le w_m$ and $w_{i, i+1} \le w_i, w_{i+1}$ for $1 \le i \le m-1$. By Corollary \ref{going-up}, $w_1$ is a $(v, \s)$-alcove element and $\<\nu_{w_1}, v\>=\<\mu, v^\flat\>$.

We claim that $w_{1, 2}$ is a $(v, \s)$-alcove element and $\<\nu_{w_{1, 2}}, v\>=\<\mu, v^\flat\>$. Indeed, since $v=p(\s)(v)$, we have $\s(M_v)=M_v$. We set $X^G_{\le_v w_1}(b_{P_v})=\cup_{w' \le_v w_1} X^G_{w'}(b_{P_v})$. By Lemma~\ref{alcove-2}, if $w' \le_v w$, then $w'$ is also a $(v, \s)$-alcove element. By Theorem~\ref{alcove-adlv} and Proposition \ref{choice}, the closed immersion $\mathcal{FL}(M_v) \to \mathcal{FL}(G)$ induces an isomorphism $X^{M_v}_{\le_v w_1}(b_{P_v}) \xrightarrow{\sim} X^G_{\le_v w_1}(b_{P_v})$. In particular, $X^G_{\le_v w_1}(b_{P_v})$ is locally a projective variety. Since the closure of $X_{w_1}(b)$ intersects $X_{w_{1, 2}}(b)$, we have $X_{\le_v w_1}(b_{P_v}) \cap X_{w_{1, 2}}(b_{P_v}) \neq \emptyset$ and hence $w_{1, 2} \le_v w_1$. Now the claim follows from Lemma \ref{nu-v} and Lemma \ref{le-v}.

Repeating the above argument, we deduce that $y$ is a $(v, \s)$-alcove element and $\<\nu_y, v\>=\<\mu, v^\flat\>$. The proof is finished.
\end{proof}

\begin{lemma}\label{str-v}
Let $x$ be a $\s$-straight element in $\Adm(\{\mu\})$ with $\dot x \in [b]$. Then there exists a unique element $v^\flat_x \in W_0 \cdot v^\flat$ such that $x$ is a $(v^\flat_x, \s)$-alcove element and $\<\nu_x, v^\flat_x\>=\<\mu, v^\flat\>$. Moreover, we have $v_x^\flat \in V^{p(\s)}$.
\end{lemma}
\begin{proof}
Since $\dot x \in [b]$, $\overline{\nu}_x=\overline{\nu}_b$. Let $y$ be the unique element in $W_0^{I(\overline{\nu}_b)}$ with $\nu_x=y(\overline{\nu}_b)$. Here $I(\overline{\nu}_b) = \{ s\in \BS_0;\ s\overline{\nu}_b = \overline{\nu}_b \}$. Set $v^\flat_x=y(v^\flat)$. Then $\<\nu_x, v^\flat_x\>=\<\overline{\nu}_b, v^\flat\>=\<\mu, v^\flat\>$.

By \cite[\S4.1 $(c) \Rightarrow (b)$]{Ni}, $x$ is a $(\nu_x, \s)$-alcove element. Since $M_{\overline{\nu}_b} \subseteq M_{v^\flat}$, we have $M_{\nu_x} \subseteq M_{v^\flat_x}$. By Lemma~\ref{alcove-1}, $x$ is also a $(v^\flat_x, \s)$-alcove element.

Since $x \in \Adm(\{\mu\})$, $x \leq t^{\mu'}$ for some $\mu' \in W_0 \cdot \mu$. By Corollary~\ref{going-up}, $t^{\mu'}$ is a $(v^\flat_x, \s)$-alcove element. In particular, $p(\s)(v^\flat_x)=v^\flat_x$.

Assume another $v \in W_0 \cdot v^\flat$ satisfies the conditions of the lemma. Then $$\<\nu_x, v\>=\<\mu, v^\flat\>=\<\nu_x, v_x^\flat\>,$$ that is, $\<\nu_x, v-v_x^\flat\>=0$. So $v-v_x^\flat$ is linear combination of coroots in $\Phi_{\nu_x, 0}^\vee$. Since $\Phi_{\nu_x} \subseteq \Phi_{v_x^\flat}$, $\<v, v\> \ge \<v_x^\flat, v_x^\flat\>$ and the equality holds if and only if $v=v_x^\flat$. On the other hand, $v, v_x^\flat$ are conjugate under $W_0$ and hence $\<v, v\>=\<v_x^\flat, v_x^\flat\>$. Therefore, we have $v=v_x^\flat$. This verifies the uniqueness of $v_x^\flat$.
\end{proof}

\begin{proposition}\label{w-alcove}
For $w \in \Adm(\{\mu\}, b)$, there exists a unique $v_w^\flat \in V^{p(\s)} \cap W_0 \cdot v^\flat$ such that $w$ is a $(v_w^\flat, \s)$-alcove element and $\<\nu_w, v^\flat_w\>=\<\mu, v^\flat\>$.
\end{proposition}

\begin{proof}
Let $C$ be a connected component of $X(\mu, b)$ such that $X_w(b) \cap C \neq \emptyset$. By \cite[Theorem 4.1]{HZ}, there exists some $\s$-straight element $x$ such that $X_x(b) \cap C \neq \emptyset$. In particular, $\dot x \in [b]$.

The existence of $v_w^\flat$ follows from this: By Lemma~\ref{str-v}, $x$ is a $(v^\flat_x, \s)$-alcove element and $\<\nu_x, v^\flat_x\>=\<\mu, v^\flat\>$. By Lemma~\ref{conn1}, $w$ is a $(v_x^\flat, \s)$-alcove element and $\<\nu_w, v^\flat_x\>=\<\mu, v^\flat\>$. We take $v^\flat_w=v^\flat_x$.

Now we prove the uniqueness.

Suppose that $w$ is a $(v, \s)$-alcove element and $\<\nu_w, v\>=\<\mu, v^\flat\>$ for some $v \in V^{p(\s)} \cap W_0 \cdot v^\flat$. By Lemma~\ref{conn1}, $x$ is a $(v, \s)$-alcove element and $\<\nu_x, v\>=\<\mu, v^\flat\>$. By Lemma~\ref{str-v}, $v=v^\flat_x$ and the uniqueness of $v^\flat_w$ follows.
\end{proof}

\subsection{Proof of Theorem~\ref{adm-adm}}
Let $w \in \Adm(\{\mu\}, b)$. By Proposition~\ref{w-alcove}, there exists $z \in W_0^J$ such that $w$ is a $(z(v^\flat), \s)$-alcove element and $\<\nu_w, z(v^\flat)\>=\<\mu, v^\flat\>$ with $p(\s)z(v^\flat)=z(v^\flat)$. Let $P'={}^z P$. Then $P'=M' N' \in \fkP^\s$. Moreover, $z_{P'}=z$. Since $w \in \Adm(\{\mu\})$, $w \le t^{y(\mu)}$ for some $y \in W_0$. By Lemma~\ref{x-sx}, $w \le_{z(v^\flat)} t^{y(\mu)}$, $t^{y(\mu)}$ is also a $(z(v^\flat), \s)$-alcove element and
$$
\<y(\mu), z(v^\flat)\>=\<\mu, v^\flat\>.
$$
Thus $\<z \i y (\mu), v^\flat\>=\<\mu, v^\flat\>$ and $\mu-z \i y (\mu) \in \BZ \Phi_{v^\flat, 0}^\vee$. Thus $z \i y(\mu)=y'(\mu)$ for some $y' \in W_{v^\flat}$. Set $y''=z y' z \i \in W_{z(v^\flat)}$. Then $y(\mu)=y'' z(\mu)$ and thus $w \le_{z(v^\flat)} t^{y(\mu)}$ implies $w \in \Adm^{M'}(\{z_{P'}(\mu)\})$.

On the other hand, suppose $P'=M' N' \in \fkP^\s$ and $w \in \Adm^{M'}(\{z(\mu)\})$. Then $w \le_{z_{P'}(v^\flat)} t^{z_{P'} u(\mu)}$ for some $u \in W_{v^\flat}$. Since $p(\s)z_{P'}(v^\flat)=z_{P'}(v^\flat)$, one checks that $t^{z_{P'} u(\mu)}$ is a $(z_{P'}(v^\flat), \s)$-alcove element. Hence by Lemma~\ref{le-v}, $w$ is a $(z_{P'}(v^\flat), \s)$-alcove element.

Let $P_i'=M_i' N_i' \in \fkP^\s$ for $i=1, 2$. Notice that $t^{z_{P_i'}(\mu)}$ is a $(z_{P_i'}(v^\flat), \s)$-alcove element and $\<\nu_{t^{z_{P_i'}(\mu)}}, z_{P_i'}(v^\flat)\>=\<\mu, v^\flat\>$. Assume there exists $x \in \Adm^{M_1'}(\{z_{P_1'}(\mu)\}) \cap \Adm^{M_2'}(\{z_{P_2'}(\mu)\})$. Then $x$ is also a $(z_{P_i'}(v^\flat), \s)$-alcove element and $\<\nu_x, z_{P_i'}(v^\flat)\>=\<\nu_{t^{z_{P_i'}(\mu)}}, z_{P_i'}(v^\flat)\>=\<\mu, v^\flat\>$. By Lemma~\ref{str-v}, $z_{P_1'}(v^\flat)=z_{P_2'}(v^\flat)$ and $z_{P_1'}=z_{P_2'}$. So $\Adm(\{\mu\}, b)$ is the disjoint union of $\Adm^{M'}(\{z_{P'}(\mu)\}, b)$ for $P'=M' N' \in \fkP^\s$.

Theorem \ref{adm-adm} is proved.

\

Now we state our main theorem in this section.

\begin{theorem}\label{HN-dec}
Assume that $(\mu, b)$ is Hodge-Newton decomposable with respect to $M$. Let $\brK$ be a standard $\s$-stable parahoric subgroup of $G(\brF)$ corresponding to $K \subseteq \tilde \BS$. Then we have the Hodge-Newton decomposition
\[
X(\mu, b)_{K}=\bigsqcup_{P'=M' N' \in \fkP^\s / W_K^\s} {\rm{Im}}(\phi_{P', K}: X^{M'}(\mu_{P'}, b_{P'})_{K_{M'}} \hookrightarrow X(\mu, b)_{K})
\]
as a disjoint union of open and closed subsets.
\end{theorem}

\begin{remark}\label{HN-dec-rmk}
In the Hodge-Newton decomposition, one may fix the standard Levi subgroup $M$ over $\brF$ and work with various rational structures on it instead of working with various semistandard Levi subgroups. More explicitly, for each $P'=M' N' \in \fkP^\s$, we have \[X^{M'}(\mu_{P'}, b_{P'})_{K_{M'}} \cong X^{M_{P'}}(\mu, \dot z_{P'}\i b_{P'} \dot z_{P'})_{K_{P'}},\] $M_{P'}(\brF)=M(\brF)$ and the Frobenius automorphism on $M_{P'}$ is given by $\s_{P'}=\Int(\dot z_{P'}\i)\circ \s \circ \Int(\dot z_{P'})$, where $K_{P'}$ is the subset of simple reflections of $\tW(M)$ corresponding to the standard parahoric subgroup $\brK_{P'}=M(\brF) \cap {}^{\dot z_{P'}\i} \brK$ of $M_{P'}(\brF)$.
\end{remark}

\begin{proof}[Proof of Theorem \ref{HN-dec}]
By Theorem~\ref{alcove-adlv} and Theorem~\ref{adm-adm}, the theorem holds for $\brK=\brI$, that is, \[ X(\mu, b)=\sqcup_{P'=M' N' \in \fkP^\s} \Im(\phi_{P'}: X^{M'}(\mu_{P'}, b_{P'}) \to X(\mu, b)).\]

By Proposition \ref{disjoint}, we deduce that \begin{align*} X(\mu, b)_{K} &= \Im(\pi_K^G)= \sqcup_{P'=M' N' \in \fkP^\s / W_K^\s} \Im(\pi^G_K \circ \phi_{P'}) \\ &= \sqcup_{P'=M' N' \in \fkP^\s / W_K^\s} \Im(\phi_{P', K} \circ \pi^{M'}_{K_{M'}}) \\ &= \sqcup_{P'=M' N' \in \fkP^\s / W_K^\s} \Im(\phi_{P', K}), \end{align*} where the first and the last equalities follow from the fact that $\pi^G_K$ and $\pi^{M'}_{K_{M'}}$ are surjective respectively (see \cite[Theorem 1.1]{He-KR}).

To show each $\Im(\phi_{P', K})$ in the finite disjoint union on the right hand side of the decomposition, is open and closed, it suffices to show it is closed. But $X^{M'}(\mu_{P'}, b_{P'})$ is closed in $\mathcal{FL}(M')$ and $\phi_{P', K}$ is a closed immersion.
\end{proof}

\section{First part of the proof}

In this section, we verify the $\to$-directions in the blueprint \S\ref{leitfaden}. By assumption, $G$ is quasi-simple over $F$. By Section~\ref{subsec:passage-to-adjoint}, we may assume that $G$ is of adjoint type. We also assume that $\mu\ne 0$.

\subsection{(1 HN-decomp) $\Rightarrow$ (2 minute)}
\label{subsec:1-implies-2}
Suppose that $\<\mu, \o_{\CO}\>+\{\<\s(0), \o_{\CO}\>\}>1$ for some $\s_0$-orbit $\CO$ of $\BS_0$. Let $v \in \bigoplus_{j \in \CO}\BR \o^\vee_j$ such that $\s_0(v)=v$ and that $\<v, \o_{\CO}\>=\<\mu, \o_{\CO}\>+\{\<\s(0), \o_{\CO}\>\}-1 > 0$. By Theorem~\ref{BG-mu}, $v \in B(G, \{\mu\})$. Note that $\Phi_{v, 0}=\Phi_{\BS_0-\CO}$ and $\mu-v \notin \BR \Phi_{\BS_0-\CO}^\vee$. Thus by definition, the $\s$-conjugacy class corresponding to $v$ is not Hodge-Newton decomposable.

\subsection{(2 minute) $\Rightarrow$ (1 HN-decomp)}
\label{subsec:2-implies-1}
Let $v \in B(G, \{\mu\})$. If $v$ is not basic, then there exists $i$ such that $\<v, \a_i\> \neq 0$. Let $\CO$ be the $\s_0$-orbit of $i$. Using Theorem~\ref{BG-mu}, we obtain
\[
\<\mu, \o_{\CO}\> + \{\<\s(0), \o_{\CO}\>\} -\<v, \o_{\CO}\> \in \BZ \quad\text{and}\quad \< \mu, \o_{\CO}\>\ge \< v, \o_{\CO}\> \ge 0.
\]
Since $\mu$ is minute, this implies that either $\<v, \o_{\CO}\>=0$ or $\<v, \o_{\CO}\>=\<\mu, \o_{\CO}\>$. Since $v$ is dominant and $\<v, \a_i\> \ne 0$, we cannot have $\<v, \o_{\CO}\>=0$.
Therefore, $\<\mu-v, \o_{\CO}\>=0$ and the corresponding $\s$-conjugacy class is Hodge-Newton decomposable with respect to $M_{\BS_0-\CO}$.

\subsection{$\text{Classification} \Rightarrow$ (2 minute)}
\label{subsec:class-implies-2}
This is just a routine check. We verify the $(\tilde B_n, \o^\vee_1, \t_1)$ case and leave the other cases to the readers.

Here $\s=\Ad(\t_1)$ and $\s(0)=\o^\vee_1$. We identify $V$ and its dual with $\bigoplus_{i=1}^n \BR e_i$, where $\<e_i, e_j\>=\d_{i j}$. The simple roots of $G$ are $e_1-e_2, e_2-e_3, \cdots, e_{n-1}-e_n, e_n$. The fundamental coweights are $\o^\vee_1=e_1, \o^\vee_2=e^\vee_1+e_2, \cdots, \o^\vee_n=e_1+e_2+\cdots+e_n$. Therefore $\<\o^\vee_1, \o_i\>=1$ and thus $\<\o^\vee_1, \o_i\>+\{\<\s(0), \o_i\>\}=1$.

\subsection{(2 minute) $\Rightarrow\text{Classification}$}
\label{subsec:2-implies-class}

By \cite[Tables]{Bour}, $\<\mu, \o_i\> \le 1$ for every $i$ if and only if the pair $(W_a, \mu)$ equals (up to isomorphism) to $(\tilde A_{n-1}, \o^\vee_1)$, $(\tilde A_{n-1}, \o^\vee_1+\o^\vee_{n-1})$, $(\tilde B/\tilde C/\tilde D, \o^\vee_1)$, $(\tilde A_3/\tilde C_2, \o^\vee_2)$.

In the last three cases, $\<\mu, \o_i\>=1$ for some $i$. By the definition of minute, we must then have $\s_0(i)=i$ and $\{\<\s(0), \o_i\>\}=0$. This leads to the choices of $\s$ for $(W_a, \mu)$ in those cases.

As to the case $(W_a, \mu)=(\tilde A_{n-1}, \o^\vee_1)$, $\<\mu, \o_i\>$ achieves its maximal value $\frac{n-1}{n}$ at $i=1$. If $\s_0$ is nontrivial, then all the possible choices of $\s$ are allowed, i.e. $\s=\t \varsigma_0$, where $\t=\t_1^k$ for
some $k$. Up to isomorphism, this leads to the two cases $(\tilde A_{n-1}, \o^\vee_1, \varsigma_0)$ and $(\tilde A_{2m-1}, \o^\vee_1, \t_1 \varsigma_0)$ in the list.

If $\s_0=\id$, i.e., $\s(0)=\t_1^k$ for $0 \le k \le n-1$, then the condition $\{\<\s(0), \o^\vee_{n-1}\>\} \le \frac{1}{n}$ implies that $k=0$ or $n-1$. This finishes the proof.

\subsection{Harris-Taylor type}
Recall that by definition $\mu$ is minute for $G$ if for any $\s_0$-orbit $\CO$ of $\BS_0$, we have $\<\mu, \o_{\CO}\>+\{\<\s(0), \o_{\CO}\>\} \le 1$.

\begin{definition}
Let $(G, \{\mu\})$ be as in Section~\ref{sec:group-theoretic-setup}; in particular in this definition we do not require $G$ to be adjoint or even semisimple, or quasi-simple.
We say that the pair $(G, \{\mu\})$ is of {\it Harris-Taylor type} if for any $\s_0$-orbit $\CO\subseteq \BS_0$, $\<\mu, \o_{\CO}\>+\{\<\s(0), \o_{\CO}\>\}<1$.
\end{definition}

\begin{lemma}
The pair $(G, \{\mu\})$ is of Harris-Taylor type if and only if $(G^\ad, \mu^\ad)$ is a product of
\begin{itemize}
\item
factors where $(W_a, \mu, \s)$ is $(\tilde A_{n-1}, \o^\vee_1, \id)$ (up to isomorphism), and
\item
factors where $\mu=0$.
\end{itemize}
\end{lemma}

\begin{proof}
We may pass to $G^\ad$ and then, as in Section~\ref{subsec:passage-to-adjoint}, pass to one of its simple factors. If $\mu\ne 0$, then the argument in \S\ref{subsec:2-implies-class} implies that the corresponding triple $(W_a, \mu, \s)$ is isomorphic to $(\tilde A_{n-1}, \o^\vee_1, \id)$.
\end{proof}

The case of $(\tilde A_{n-1}, \o^\vee_1, \id)$ is the case considered by Harris and Taylor in \cite{HT}.

In this case, it is easy to check that every element in $\Adm(\{\mu\})$ is straight. Therefore, for any $[b] \in B(G, \{\mu\})$, and any parahoric subgroup $\brK$, $\dim X(\mu, b)_K=0$.

\subsection{(1 HN-decomp)+(2 minute) $\Rightarrow$ (3 non-basic --leaves)}
Recall that $G$ is assumed to be quasi-simple over $F$. In particular, $\s$ and $\s_0$ act transitively on the connected components of $\tilde \BS$ and $\BS_0$ respectively.

\begin{lemma} \label{non-zero}
Let $\mu$ be a non-zero dominant coweight. Then $\mu^\diamond \in \sum_{i \in \BS_0} \BR_{>0} \a_i^\vee$. Recall that $\mu^\diamond=\frac{1}{n_0}\sum_{i=0}^{n_0-1} \s_0^i(\mu)$ with $n_0$ the order of $\s_0$.
\end{lemma}
\begin{proof}
Let $\tilde \BS'$ be a connected component of $\tilde \BS$ such that $\mu$ is nonzero on the connected Dynkin diagram $\BS_0':=\tilde \BS' \cap \BS_0$.
It is already known that $\mu \in \sum_{j \in J} \BR_{>0} \a_j^\vee$ for some $J \subseteq \BS_0$. Since $\mu$ is nonzero on $\BS_0'$, we have $J \cap \BS_0' \neq \emptyset$. We claim that $\BS_0' \subseteq J$. Otherwise, since $\BS_0'$ is connected, there exists $i_0 \in \BS_0' - J$ such that $\<\a_{j_0}^\vee, \a_{i_0}\> < 0$ for some $j_0 \in J$. Then one checks that $\<\mu, \a_{i_0}\> < 0$, contradicting the assumption $\mu$ is a dominant coweight. So the claim is proved. Now the lemma follows since $\s_0$ acts transitively on the connected components of $\BS_0$.
\end{proof}

\label{subsec:2-implies-3}
Let $[b]$ be a nonbasic $\s$-conjugacy class in $B(G, \{\mu\})$. Then by assumption, $[b]$ is Hodge-Newton decomposable with respect to a proper standard Levi subgroup $M$ with $\s_0(\BS_M)=\BS_M$, where $\BS_M$ denotes the set of simple reflections of $W_0(M)$.

We have a direct sum decomposition $V=\BR \Phi_M^\vee \oplus \bigoplus_{j \notin \BS_M} \BR \o^\vee_j$. Let $\mu_M \in \BR \Phi_M^\vee$ be the first factor of $\mu$ with respect to this decomposition. Since $\mu_M$ is $M$-dominant, $\mu_M \in \sum_{i \in \BS_M} \BR_{\ge 0} \a_i^\vee$ and thus $\<\mu-\mu_M, \a_j\> \ge 0$ for $j \in \BS_0 -\BS_M$, which means $\mu-\mu_M$ is dominant. Since $\mu \neq 0$ and $M \subsetneq G$, applying Lemma \ref{non-zero}, we have $\mu-\mu_M \neq 0$ and hence

($\dagger$) $\mu^\diamond-\mu_M^\diamond \in \sum_{i \in \BS_0} \BR_{>0} \a_i^\vee$.

For $i \in \BS_M$, let $\o^M_i \in \BR \Phi_M$ be the fundamental weight of $M$ corresponding to the simple root $\a_i$. For each $\s_0$-orbit $\CO$ in $\BS_M$ we set $\o_\CO^M=\sum_{i \in \CO} \o_i^M$.

\

Let $P'$, $z_{P'}$ and $\s_{P'}=\Int(\dot z_{P'}\i) \circ\s\circ \Int(\dot z_{P'})$ be as in Remark~\ref{HN-dec-rmk}.
\begin{lemma} \label{frack}
We have $\{\<\s_{P'}(0), \o^M_{\CO}\>\}=\{\<\s(0), \o_{\CO}\>\}$ for each $\s_0$-orbit $\CO$ in $\BS_M$.
\end{lemma}
\begin{proof}
Assume $\s \in t^\l W_0 \s_0$ for some coweight $\l$. Then $\s_{P'}=t^{z_{P'}\i(\l)} w \s_0$ for some $w \in W_0$. Write $z_{P'}\i(\l)=r+h$ with $r \in \BR \Phi_M^\vee$ and $h \in \sum_{j \in \BS_0-\BS_M} \BR \o_j^\vee$. Let $n \in \BN$ be the order of $\s_{P'}$ as an affine transformation of $V$. In particular, $(\s_0)^n=1$. We show that

($\ddagger$) $\sum_{k=0}^{n-1} \s_0^k(z_{P'}\i(\l)) \in \sum_{j \in \BS_M} \BR \a_j^\vee$.

This is \cite[Lemma 4.2 (1)]{HN2}. We repeat the proof here for convenience. Let $v^\flat$ be as in \S \ref{HN-setup}. Then $P'=P_{z_{P'}(v^\flat)}$. By assumption we have ${}^\s (P')=P'$ and hence $p(\s_{P'})(v^\flat)=v^\flat$. Since $v^\flat$ is dominant and $v^\flat = p(\s_{P'})(v^\flat)=w \s_0(v^\flat)=w(v^\flat)$, we have $w \in W_{\BS_M}$. Then the equality $\sum_{k=0}^{n-1} (w \s_0)^k(z_{P'}\i(\l))=0$ (since $(\s_{P'})^n=1$) implies $\sum_{k=0}^{n-1} (w\s_0)^k(h)=\sum_{k=0}^{n-1} \s_0^k(h)=0$. Thus ($\ddagger$) follows.

By ($\ddagger$) we have $$ \<\s_{P'}(0), \o_\CO^M\> =\<z_{P'}\i(\l), \o_\CO^M\> =\<z_{P'}\i(\l), \o_\CO\> \in \<\l, \o_\CO\> +\BZ=\<\s(0), \o_\CO\>+\BZ,$$ where the inclusion follows from the fact $z_{P'}\i(\l)-\l \in \BZ \Phi^\vee$. The proof is finished.
\end{proof}

Let $\CO$ be a $\s_0$-orbit of $\BS_M$. Combining ($\dagger$) and Lemma \ref{frack} we deduce that $$\<\mu_M, \o^{M}_{\CO}\>+\{\<\s_{P'}(0), \o^M_{\CO}\>\}=\<\mu, \o^M_{\CO}\>+\{\<\s(0), \o_{\CO}\>\} < \<\mu, \o_{\CO}\>+\{\<\s(0), \o_{\CO}\>\} \le 1.$$ Therefore $(M_{P'}, \mu)$ is of Harris-Taylor type and $\dim X^{M_{P'}}(\mu, \dot z_{P'}\i b_{P'} \dot z_{P'})_{K_{P'}}=0$. Now (3 non-basic -- leaves) follows from the Hodge-Newton decomposition Theorem \ref{HN-dec}.

\subsection{}
To prepare for the proof of the implication (3 non-basic -- leaves) $\Rightarrow$ (4 EKOR $\subseteq$ Newton) in Theorem~\ref{main}, we need the following

\begin{proposition}\label{zero-diml-adlv}
Let $w \in \tW$ and $b \in G(\brF)$. Then the following conditions are equivalent:

(1) $\dim X_w(b)=0$;

(2) The $\s$-centralizer $\BJ_b$ of $b$ acts transitively on $X_w(b)$;

(3) $w$ is a $\s$-straight element with $\dot w \in [b]$.
\end{proposition}

\begin{proof}
$(2)\Rightarrow(1)$: This follows from the fact that the group $\BJ_b$, i.e., the group of $F$-valued points of the $\s$-centralizer of $b$, considered as a subscheme of the loop group $LG$ over $\kk$ (considered as a perfect scheme as in Zhu's setting~\cite{Zhu}, if $\mathop{\rm char}(F)=0$), has dimension zero.

$(3)\Rightarrow(2)$: This follows from \cite[Theorem 4.8]{He14}.

$(1)\Rightarrow(3)$: Since $X_w(b) \ne \emptyset$, we have $\dot w \in [b]$. If $w$ is not a minimal length element in its $\s$-conjugacy class, then there exists $w' \in \tW$ and $s \in \tSS$ such that $w \approx_\s w'$ and $s w' \s(s)<w'$. Then by Deligne-Lusztig reduction, $$\dim X_w(b)=\dim X_{w'}(b)=\max\{\dim X_{s w'}(b), \dim X_{s w' \s(s)}(b)\}+1.$$ Since $X_w(b) \neq \emptyset$, $X_{s w'}(b) \neq \emptyset$ or $X_{s w' \s(s)}(b) \neq \emptyset$. Hence $\dim X_w(b) \ge 1$.

If $w$ is a minimal length element in its $\s$-conjugacy class, then by \cite{He14}, $f(b)=f(w)$ and so $0 = \dim X_w(b)=\ell(w)-\<\overline{\nu}_b, 2 \rho\>$, hence $w$ is $\s$-straight.
\end{proof}

\subsection{(3 non-basic -- leaves) $\Rightarrow$ (4 EKOR $\subseteq$ Newton)}
\label{subsec:3-implies-4}

By \cite[Proposition 6.20]{HR}, for any $b$ we have $$\dim X(\mu, b)_K=\max_{w \in \EO{K}} \dim X_w(b).$$

Let $w \in \EO{K}$. If $X_w(b) \neq \emptyset$ for some nonbasic element $b$, then by (3 non-basic -- leaves), $\dim X_w(b)=0$. Thus by Proposition \ref{zero-diml-adlv}, $w$ is a $\s$-straight element with $\dot w \in [b]$. By Theorem \ref{s-newton}, $[b]$ is the unique $\s$-conjugacy class such that $X_w(b) \neq \emptyset$.

Otherwise, $X_w(b) \neq \emptyset$ only for basic $b$. Since $[b] \in B(G, \{\mu\})$, it is the unique minimal element $[b_0]$.

\subsection{Support of $w$}
For $w \in W_a$, we denote by $\supp(w)$ the support of $w$, i.e., the set of $i \in \tSS$ such that $s_i$ appears in some (or equivalently, any) reduced expression of $w$. For $\t \in \Omega$, define
\[
\supp_\s(w \t)=\bigcup_{n \in \BZ} (\t \s)^n(\supp(w)).
\]
In other words, $\supp_\s(w \t)$ is the minimal $\t \s$-stable subset $K'\subseteq\tSS$ such that $w \t \s\in W_{K'} \rtimes \<\t \s\>$.

Let $w \in \tW$ and $K \subseteq \tSS$. We write $\Ad(w) \s(K)=K$ if for any $k \in K$, there exists $k' \in K$ with $w s_{\s(k)} w \i=s_{k'}$. In this case, $w \in {}^K \tW^{\s(K)}$.

As in \cite[3.1]{GH}, for any $w \in \tW$ and $K \subseteq \tSS$ with $\s(K)=K$, the set $\{K' \subseteq K; \Ad(w)\s(K')=K'\}$ contains a unique maximal element, which we denote by $I(K, w, \s)$, cf.~also \cite[1.3]{He:Gstable}.

We have the following results on the finiteness of support.

\begin{proposition}\label{finite}
Let $w \in \tW$. Then $\brI w \brI  \subseteq [b_0]$ if and only if $\k(w)=\k(b_0)$ and $W_{\supp_\s(w)}$ is finite.
\end{proposition}

\begin{proof}
Write $w=w_1 \t$ with $w_1 \in W_a$ and $\ell(\t)=0$. As $\k(\t)=\k(w)=\k(b_0)$, we have $[\dot \t] = [b_0]$. So we may assume $b_0 = \dot \t$.

First we show the ``if'' part. Suppose $W_{\supp_\s(w)}$ is finite. Denote by $\brP \subseteq G(\brF)$ the corresponding standard parahoric subgroup. Then $w_1 \in W_{\supp_\s(w)}$ and $\brP$ is stable under the (Frobenius) automorphism given by $\g \mapsto \dot\t \s(g) \dot\t\i$. By Lang's theorem, we have $\brP = \{g\i \dot\t \s(g) \dot\t\i; g \in \brP\}$. In other words, $\brP \dot\t \subseteq [\dot \t]$. This implies $\brI w \brI = \brI w_1 \t \brI = \brI w_1 \brI \t \subseteq \brP \t \subseteq [\dot \t] = [b_0]$ as desired.

Now we prove the ``only if'' part. Since $\brI w \brI  \subseteq [b_0]$, we have  $\k(w)=\k(b_0)$. Suppose that $W_K$ is not finite. Then $K$ contains at least one connected component $X$ of $\tSS$. By definition, a reduced expression of $w \t \i$ contains at least one simple reflection in each $\t \s$-orbit of $\cup_{n \ge 0} (\t \s)^n X$. Let $c$ be the product of these simple reflections in the same order as they first appear in a fixed reduced expression of $w \t \i$. Then $c \t \le w$. Therefore $\dot c \dot \t$ is contained in the closure of $\brI w \brI$. By \cite[Theorem 2.40]{He-CDM}, $\overline{\nu}_{\dot c \dot \t} \le \overline{\nu}_{b_0}$. Since $[b_0]$ is basic, we have $\overline{\nu}_{\dot c \t}=\overline{\nu}_{b_0}$. By \cite[Proposition 3.1]{HN}, $c \t$ is a $\s$-straight element. Therefore $\overline{\nu}_{\dot c \dot \t} \neq \overline{\nu}_{b_0}$. That is a contradiction.

Note that although we need to talk about the closure of an Iwahori double coset, working inside the $p$-adic loop group/affine Grassmannian, this proof does work in mixed characteristic. Alternatively, instead of invoking this theory, one can also apply the above proof in the equal characteristic setting, working with the ``usual'' loop group, and then rephrase the statement of the proposition in a purely combinatorial way, using class polynomials, say, in order to show that the result can be transferred from the equal to the mixed characteristic case.
\end{proof}

\begin{proposition}\label{finite-1}
Let $K \subseteq \tSS$ with $W_K$ finite. Let $w \in {}^K \tW$. If $W_{\supp_\s(w)}$ is finite, then $W_{\supp_\s(w) \cup I(K, w, \s)}$ is also finite.
\end{proposition}
\begin{proof}
Let $x \in W_{I(K, w, \s)}$. Similarly to the proof of Proposition \ref{finite}, for the parahoric subgroup $P_{I(K, w, \s)}$ we have $\brI w x \brI \subseteq G(\brF) \cdot_\s \brI  w \brI$. Since $W_{\supp_\s(w)}$ is finite, $\brI  w \brI$ is contained in a basic $\s$-conjugacy class $[b_1]$ for some basic element $b_0 \in G(\brF)$ with $\k(b_0) = \k(w)$. Thus $\brI w x \brI \subseteq [b_0]$. By Proposition~\ref{finite}, $W_{\supp_\s(w x)}$ is finite for any $x \in W_{I(K, w, \s)}$. It follows that $W_{\supp_\s(w) \cup I(K, w, \s)}$ is finite.
\end{proof}

\subsection{(4 EKOR $\subseteq$ Newton) $\Rightarrow$ (5 basic -- EKOR)}
\label{subsec:4-implies-5}
We start with a criterion on the finiteness of support.

\begin{lemma}\label{fix}
Let $w \in \tW$. Then $W_{\supp_\s(w)}$ is finite if and only if the action of $w \circ \s$ fixes a point in the closure $\bar \aa$ of the fundamental alcove $\aa$.
\end{lemma}

\begin{proof}
Say we have $w \in W_a\t$ for $\t \in \Omega$.

($\Rightarrow$): Let $K=\supp_\s(w)$. Then $w\t\i\in W_K$ and by assumption $W_K$ is finite. We have $V^K \cap \bar \aa \neq \emptyset$, where $V^K=\{v \in V; s(v)=v \text{ for all } s \in K\}$. Since $\t\s(K)=K$, the action of $\t\s$ on $V$, which is of finite order, preserves $V^K \cap \bar \aa$. So there exists $v \in V^K \cap \bar \aa$ such that $\t\s(v)=v$ and hence $w\s(v)=v$ as desired.

($\Leftarrow$): Let $v \in \bar \aa$ such that $w\s(v)=v$ and let $K=\{s \in \tSS; s(v)=v\}$. In particular, $W_K$ is finite. Since $v \in \bar \aa \cap w\s (\bar \aa)$, there exists $w' \in W_K$ such that $w\s(\aa)= w'(\aa)$. So $\t\s = (w')\i w\s$ and hence $\t\s(v)=v$. Therefore, we have $\t\s(K)=K$ and $\supp_\s(w) \subseteq K$ as desired.
\end{proof}
 The implication (4 EKOR $\subseteq$ Newton) $\Rightarrow$ (5 basic -- EKOR) in Theorem~\ref{main} follows immediately from Prop.~\ref{finite} and Lem. \ref{fix}.

Using Prop.~\ref{finite-1}, we also obtain the informal statement we used in Theorem~\ref{thmB} in the introduction: The space $X(\mu, b)_K$ is naturally a union of classical Deligne-Lusztig varieties.

In fact, let $w \in \EO{K}$ with $X_w(b_0) \neq \emptyset$. By (4 EKOR $\subseteq$ Newton) $\brI w \brI \subseteq [b_0]$. By Proposition \ref{finite} and Proposition \ref{finite-1}, $W_{\supp_\s(w) \cup I(K, w, \s)}$ is finite. By \cite[Proposition 4.1.1 \& Theorem 4.1.2]{GH}, $X_{K, w}(b_0)$ is naturally a union of classical Deligne-Lusztig varieties.
The Deligne-Lusztig varieties occurring here arise as translates of Deligne-Lusztig varieties in (the perfection of) the classical partial flag variety, embedded in the $p$-adic partial affine flag variety (the quotient $\brP/\brP'$ of any two parahoric subgroups $\brP'\subset \brP$ can be identified with the perfection of the corresponding quotient (a ``partial flag variety'') of the maximal reductive quotient of the special fiber of $\brP$; this works out in the same way as the special case where $\brP$ is hyperspecial and $\brP'$ corresponds to a maximal parabolic subgroup, see map~(1.4.4) and Cor.~1.24 in \cite{Zhu}). Therefore this decomposition can be viewed as a disjoint union of perfect schemes, where of course we have to pass to the perfection of the Deligne-Lusztig varieties involved. Now the informal version of (5) follows from the decomposition in \S\ref{fine},
\[
X(\mu, b_0)_K=\bigsqcup_{w \in \EO{K}; X_w(b_0) \neq \emptyset} X_{K, w}(b_0).
\]

Conversely, we show the informal version of (5) implies (5 basic -- EKOR). Indeed, the classical Deligne-Lusztig varieties are naturally in the partial flag variety of the reductive quotient of some $\Ad(\t) \circ \s$-stable standard parahoric subgroup of type $P \subseteq \tSS$. In particular, the Weyl group elements $w_1$ associated to these classical Deligne-Lusztig varieties are contained in $W_P$. So each $w \in \EO{K}$ with $X_w(b_0) \neq \emptyset$ is of the form $w_1 \t$ for some $w_1 \in W_P$. This means $\supp_\s(w)=\cup_{i \in \BZ} (\Ad(\t)\circ \s)^i \supp(w_1) \subseteq P$ and hence $W_{\supp_\s(w)}$ is finite. Now (5 basic -- EKOR) follows from Lem. \ref{fix}.

\subsection{Closure relations}

First, note the following fact

\begin{proposition}
The decomposition of $X(\mu, b)_K$ into (perfections of) classical Deligne-Lusztig varieties described in the previous section is a stratification in the sense that the closure of each stratum is a union of strata.
\end{proposition}

\begin{proof}
In view of the fact that the decomposition arises by projection from the decomposition in the Iwahori case, and since this projection is closed, it is enough to observe that the closure of a classical Deligne-Lusztig variety is a union of classical Deligne-Lusztig varieties.
\end{proof}

A combinatorial description of the closure relations in this decomposition can in principle, at least in each individual case, be worked out along the lines of ~\cite[\S 7]{GH}, cf.~also \cite[Theorem 3.1 (2)]{He:Gstable}, where the closure of a general fine Deligne-Lusztig variety is described. However, the combinatorics involved is in general much more complicated than in the Coxeter cases of \cite{GH}, e.g., the set $\Adm(\{\mu\}) \cap {}^K\tW$ will not in general have a simple description.

Although formally speaking, in \cite{GH} only the case of equal characteristic was considered, the decomposition and the same closure relations hold in the mixed characteristic case, using once again \cite[Cor. 1.24]{Zhu} and the fact that the projections from the affine flag variety to partial affine flag varieties are closed in this setting, too.

\section{The second part of the proof}
\label{sec:fc-implies-class}

\subsection{} Now we introduce the following condition, formulated in terms of the affine Weyl group $\tW$ (together with the action of $\s$ and the choices of $\mu$ and $K$). This is condition (6) in Thm.~\ref{main}.

\textbf{Fixed point condition (FC)}: If $w \in \Adm(\{\mu\}) \cap {}^K \tW$ with central Newton point, then $w \circ \s$ fixes a point in $\bar \aa$.

\subsection{(5 basic -- EKOR) $\Rightarrow$ (6 FC)}\label{5-to-FC}

Let $w \in \EO{K}$. If $\overline{\nu}_w$ is central, then $\dot w$ lies in the basic $\s$-conjugacy class $[b_0]$. By condition (5 basic -- EKOR) in the main theorem Theorem~\ref{main}, $w$ satisfies the fixed point condition.

\smallskip

In the rest of this section, we show that the fixed point condition (FC) implies Theorem \ref{main} (2 minute), or equivalently, the classification.

\subsection{} To study the (FC) condition, it suffices to consider each $F$-factor of the corresponding group of adjoint type, $G^{\ad}$, individually. Our first goal is to reduce the study of the (FC) condition further to groups that are simple over $\brF$.

Assume that $G$ is adjoint and simple over $F$. Let $(G', \{\mu'\})$ be the $\brF$-simple pairs associated to $(G, \{\mu\})$ (Section~\ref{passage-brF-simple}). Let $\brK=(\brK_1, \cdots, \brK_k)$ be a $\s$-stable parahoric subgroup of $G(\brF)$. Then $\brK'=\brK_k$ is a $\s'$-stable parahoric subgroup of $G'(\brF)$. In particular, our choice of a standard Iwahori subgroup of $G$ gives rise to a standard Iwahori subgroup of $G'$.

We have the following additivity on the admissible sets. It is proved in \cite[Theorem 5.1]{He-KR} for Iwahori subgroups and in \cite[Theorem 1.4]{HH} in the general case.

\begin{theorem}\label{thm-add}
Let $\l_1, \l_2$ be two dominant cocharacters. Then $$\Adm^K(\{\l_1\}) \Adm^K(\{\l_2\})=\Adm^K(\{\l_1+\l_2\}).$$
\end{theorem}

As a consequence, we have

\begin{proposition}
The projection map to the $k$-th factor $G(\brF)/\brK \to G'(\brF)/\brK'$ induces a surjection $X^G(\mu, b)_K \to X^{G'}(\mu', b')_{K'}$.
\end{proposition}

\begin{proof}
We have
\begin{align*} & \brK_k \Adm(\{\mu_k\}) \brK_k \s (\Adm(\{\mu_{k-1}\})) \brK_k \cdots \brK_k \s^{k-1} (\Adm(\{\mu_1\})) \brK_k \\ &=\brI' \Adm^{K'}(\{\mu_k\}) \brI' \s(\Adm^{K'}(\{\mu_{k-1}\})) \brI' \cdots \brI' \s^{k-1} (\Adm^{K'}(\{\mu_1\})) \brI' \\ &=\brI' \Adm^{K'}(\{\mu_k\})  \s(\Adm^{K'}(\{\mu_{k-1}\}))  \cdots \s^{k-1} (\Adm^{K'}(\{\mu_1\})) \brI'  \\ &=\brI' \Adm^{K'}(\{\mu_k\}) \Adm^{K'}(\{\s_0(\mu_{k-1})\}) \cdots \Adm^{K'}(\{\s_0^{k-1}(\mu_1)\}) \brI' \\ &=\brI' \Adm^{K'}(\{\mu'\}) \brI'=\brK' \Adm(\{\mu'\}) \brK'.
\end{align*}
Here $\brI'$ is the standard Iwahori subgroup of $G'$, and the second equality follows from the fact that the admissible sets are closed under the Bruhat order; the third one follows from the fact that $\{p(\s)(\l)\}=\{\s_0(\l)\}$ for any cocharacter $\l$; the fourth one follows from Theorem \ref{thm-add}.

Let $b=(b_1, \cdots, b_k)$ and $(g_1, \cdots, g_k) \brK \in X^G(\mu, b)_K$. Then \begin{gather*} g_k \i b_k \s(g_{k-1}) \in \brK_k \Adm(\{\mu_k\}) \brK_k,\\ g_{k-1} \i b_{k-1} \s(g_{k-2}) \in \brK_{k-1} \Adm(\{\mu_{k-1}\}) \brK_{k-1}, \\ \cdots, \\ g_1 \i b_1 \s(g_k) \in \brK_1 \Adm(\{\mu_1\}) \brK_1.\end{gather*} Therefore for $b'=b_k \s(b_{k-1}) \cdots \s^{k-1}(b_1)$,
\begin{align*} g_k \i b' \s'(g_k) & \in \brK_k \Adm(\{\mu_k\}) \brK_k \s (\Adm(\{\mu_{k-1}\})) \brK_k \cdots \brK_k \s^{k-1} (\Adm(\{\mu_1\})) \brK_k \\ &=\brK' \Adm(\{\mu'\}) \brK'.\end{align*} Hence $g_k \brK' \in X^{G'}(\mu', b')_{K'}$.

On the other hand, if $g_k \brK' \in  X^{G'}(\{\mu'\}, b')_{K'}$, then \begin{align*} g_k \i b' \s'(g_k) & \in \brK' \Adm(\{\mu'\}) \brK' \\ &=\brK_k \Adm(\{\mu_k\}) \brK_k \s (\Adm(\{\mu_{k-1}\})) \brK_k \cdots \brK_k \s^{k-1} (\Adm(\{\mu_1\})) \brK_k.\end{align*} Therefore, there exist $g_1, g_2, \cdots, g_{k-1}$ such that $(g_1, \cdots, g_k) \brK \in X^G(\mu, b)_K$.
\end{proof}

\subsection{}
It is easy to see that $\mu$ is minute for $G$ if and only if $\mu'$ is minute for $G'$. On the other hand, let
$$m: \Adm^K(\{\mu\}) \to \Adm^{K'}(\{\mu'\}),\quad (w_1, \cdots, w_k) \mapsto w_k \s(w_{k-1}) \cdots \s^{k-1}(w_1)$$ be the multiplication map. There is no obvious relation between the image of $\Adm(\{\mu\}) \cap {}^K \tW$ (under $m$) and $\Adm(\{\mu'\}) \cap {}^{K'} \tW'$, and it is difficult to relate the (FC) conditions of $G$ and of $G'$ directly.

To overcome this difficulty, we introduce the subset $${}^K \Adm(\{\mu\})_\spadesuit=\bigcup_{\l \in W_0 \cdot \mu} ({}^K \tW \cap t^\l W_K).$$

By \cite[Theorem 6.1]{He-KR} (see also \cite[Proposition 5.1]{HH} for a different proof), we have that ${}^K \Adm(\{\mu\})_\spadesuit \subseteq \Adm(\{\mu\})$.

\begin{figure}

\begin{tikzpicture}[scale=.8]

\draw[fill=lightgray] (-2, -2) -- (4, 4) -- (4, -2) -- cycle;
\draw[fill=gray] (-2, -2) -- (0, 0) -- (0, -2) -- cycle;
\draw[fill=gray] ( 2, 0) -- (4,0) -- (4, -2) -- (2, -2) -- cycle;
\draw[fill=gray] ( 2,  2) -- (4, 4) -- (4, 2) -- cycle;

\draw[line width=1mm] (-2, 2) -- (0, 4) -- (2, 2) -- (4, 4) -- (4, -2) -- (-2, -2) -- (0, 0) -- cycle;

\draw (-5,-4) -- (5,-4);
\draw (-5,-2) -- (5,-2);
\draw (-5,0) -- (5,0);
\draw (-5,2) -- (5,2);
\draw (-5,4) -- (5,4);

\draw (-4,-5) -- (-4,5);
\draw (-2,-5) -- (-2,5);
\draw (0,-5) -- (0,5);
\draw (2,-5) -- (2,5);
\draw (4,-5) -- (4,5);

\draw (-4.7, -4.7) -- (4.7, 4.7);
\draw (-4.7, 4.7) -- (4.7, -4.7);
\draw (-4.7, -0.7) -- (0.7, 4.7);
\draw (-0.7, -4.7) -- (4.7, 0.7);
\draw (-0.7, 4.7) -- (4.7, -0.7);
\draw (-4.7, 0.7) -- (0.7, -4.7);
\draw (-4.7, -3.3) -- (-3.3, -4.7);
\draw (4.7, 3.3) -- (3.3, 4.7);
\draw (-4.7, 3.3) -- (-3.3, 4.7);
\draw (4.7, -3.3) -- (3.3, -4.7);
\draw[fill, black] (0,0) circle [radius=.17cm];
\end{tikzpicture}

\caption{An example for $G=PSp_4$, $\mu = (\frac 12, \frac 12)$, $K= \{s_1 \}$. The origin is at the thick dot. The set $\Adm(\{\mu\})$ consists of the thirteen alcoves inside the thick line. The intersection ${}^K \tW \cap \Adm(\{\mu\})$ consists of the $9$ alcoves shaded in (light or dark) gray. The set ${}^K\Adm(\{\mu\})_\spadesuit$ consists of the $4$ alcoves shaded in dark gray.}
\end{figure}
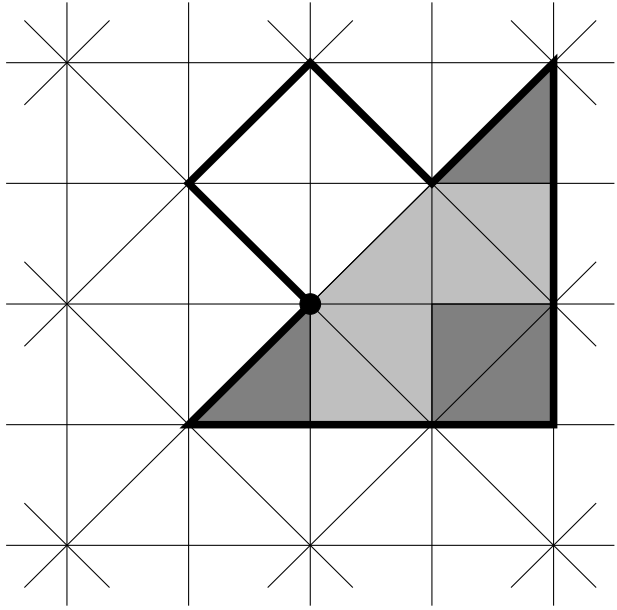

\begin{proposition}\label{weak-add}
Suppose $W_K$ is finite. Let $\mu_1, \mu_2 \in X_*(T)_{\G_0}$ be dominant cocharacters. Then
$${}^K \Adm(\{\mu_1+\mu_2\})_\spadesuit \subseteq {}^K \Adm(\{\mu_1\})_\spadesuit \ {}^K \Adm(\{\mu_2\})_\spadesuit.$$
\end{proposition}

\subsection{}\label{add} We first consider the case where $K \subseteq \BS_0$. In this case, it is easy to see that if $t^\l w \in {}^K \tW$ with $w \in W_K$, then $\l$ is $K$-dominant, i.e., $t^\l \in {}^K \tW$.

Let $y \in {}^K \Adm(\{\mu_1+\mu_2\})_\spadesuit$. We may write $y$ as $y=t^\l u$, where $\l \in \{\mu_1+\mu_2\}$, the $W_0$-conjugacy class of $\mu_1+\mu_2$, and $u \in W_K$. Since $\l$ is $K$-dominant, $\l=\l_1+\l_2$, where $\l_i \in \{\mu_i\}$ for $i=1, 2$ are $K$-dominant. We set $K_1=\{i \in K; \<\l_2, \a_i\>=0\}$. Write $u=u_1 u_2$ with $u_1 \in W_{K_1}$ and $u_2 \in {}^{K_1} W_K$. Then $t^{\l_2} u_2 \in {}^K \tW$ and $y=t^{\l_1+\l_2} u=(t^{\l_1} u_1) (t^{\l_2} u_2)$. It remains to show that
\[
t^{\l_1} u_1 \in {}^K \tW.
\]
We prove this by contradiction. Suppose that $t^{\l_1} u_1 \notin {}^K \tW$. Then there exists $\a \in \Phi^+_{K_1}$ such that $u_1 \i (\a)<0$ and $\<\l_1, \a\>=0$. Note that $\<\l_2, \a\>=0$. Since $-u_1 \i (\a) \in \Phi^+_{K_1}$ and $u_2 \in {}^{K_1} W_K$, we have $u_2 \i u_1 \i(\a)<0$. Thus $u(\a)<0$ and $\<\l, \a\>=0$. Hence $t^\l u \notin {}^K \tW$. That is a contradiction.

This finishes the proof for the case where $K \subseteq \BS_0$.

\subsection{} \label{reduce} The general case involves the comparison between the original root system and the one defined by a ``change of origin''. To explain this, we introduce the following notation.

Let $p: \text{Aff}(V)=V \rtimes GL(V) \to GL(V)$ be the projection map. We set $W_{\uK}=p(W_K) \subseteq GL(V)$. This is the Weyl group associated to the root system $\Phi_{\uK}$ defined as follows:
Let $i \in K$. If $i \in \BS_0$, we set $\ua_i=\a_i$. Otherwise, we set $\ua_i=-\th$, where $\th \in \Phi^+$ is the highest root such that $s_i=t^{\th^\vee} s_\th$. Denote by $\Phi_{\uK}$ the (finite) root system spanned by $\ua_i$ for $i \in K$. This is a root system of Dynkin type $K$. We see that $\Phi_{\uK}^+:=\Phi \cap (\sum_{i \in K} \BN \ua_i)$ is a system of positive roots of $\Phi_{\uK}$. Let $\tW_{\uK}=X_*(T)_{\G_0} \rtimes W_{\uK}$. We may regard $\tW_{\uK}$ as the subset $X_*(T)_{\G_0} W_{\uK}=X_*(T)_{\G_0} W_K$ of $\tW$ in the natural way. The length function $\ell$ and the Bruhat order $\le$ on $\tW$ restrict to a $\BN$-valued map and a partial order on $\tW_{\uK}$.

On the other hand, $\tW_{\uK}$ is an extended affine Weyl group. Let $\ell_{\uK}$ (resp. $\le_{\uK}$) be the length function (resp. Bruhat order) determined by the base alcove $$\aa_{\uK}:=\{v \in V;\ 0 < \<v, \ua\> < 1,\ \ua \in \Phi_{\uK}^+\}$$ in the usual way. Following our usual conventions, we write ${}^{\uK} \tW_{\uK} := {}^{\uK}(\tW_{\uK})$ for the system of minimal length representatives in $\tW_{\uK}$ for the cosets in $W_{\uK}\backslash \tW_{\uK}$.

To compare the two functions $\ell$ and $\ell_{\uK}$, and the two partial orders $\le$ and $\le_{\uK}$ on $\tW_{\uK}$, we need to change the origins between $0$ in $V$ and the barycenter $e^K$ of the facet corresponding to $\tSS-K$.
Let $t^{e^K}: V \to V$ be the affine transformation given by translation by $e^K$. Then $t^{-e^K} \tW_{\uK} t^{e^K}=\tW_{\uK}$ and $t^{-e^K}(\aa) \subseteq \aa_{\uK}$. Therefore, for $\tw, \tw' \in \tW_{\uK}$, we have

(a) $\tw \leq \tw'$ if $t^{-e^K} \tw t^{e^K} \leq_{\uK} t^{-e^K} \tw' t^{e^K}$.

(b) $\tw \in {}^K \tW$ if and only if $t^{-e^K} \tw t^{e^K} \in {}^{\uK} \tW_{\uK}$.

Note that if $K \subseteq \BS_0$, then the conjugation action of $t^{-e^K}$ on $\tW_{\uK}$ is trivial, but if $K$ contains $s_0$, then the conjugation action of $t^{-e^K}$ on $\tW_{\uK}$ is nontrivial.

\subsection{Proof of Proposition \ref{weak-add}}
Let $w \in {}^K \Adm(\{\mu_1+\mu_2\})_\spadesuit$. Set $y=t^{-e^K} w t^{e^K} $. By \ref{reduce}~(b), $y \in {}^{\uK} \tW_{\uK}$. So there exists $x \in W_0$ such that $x(\mu_1+\mu_2)$ is $K$-dominant and $y \in {}^{\uK} \tW_{\uK} \cap t^{x(\mu_1+\mu_2)} W_K$. Since $\mu_1, \mu_2$ are dominant, we see that $x(\l_1), x(\l_2)$ are also $K$-dominant. By \S\ref{add}, there exist $y_i \in ({}^{\uK} \tW_{\uK} \cap t^{x(\mu_i)}  W_{\uK})$ for $i \in \{1, 2\}$ such that $y=y_1y_2$. By \ref{reduce}~(b), the element $w_i:=t^{e^K} y_i t^{-e^K}$ lies in ${}^K \Adm(\{\mu_i\})_\spadesuit$ for $i \in \{1, 2\}$. The statement now follows from the equality $w=w_1 w_2$. This concludes the proof of Proposition~\ref{weak-add} in the general case.

\subsection{} As a consequence of Proposition \ref{weak-add}, ${}^{K'} \Adm(\{\mu'\})_\spadesuit$ is contained in the image of ${}^K \Adm(\{\mu\})_\spadesuit$ under the multiplication map $m: \Adm^K(\{\mu\}) \to \Adm^{K'}(\{\mu'\})$.

We introduce the following weaker version of the condition (FC):

\textbf{(FC$_\spadesuit$)}: If $w \in {}^K \Adm(\{\mu\})_\spadesuit$ has central Newton point, then $w \circ \s$ fixes a point in $\bar \aa$.

This new condition behaves better than (FC) with respect to passing from $G$ to $G'$: If (FC$_\spadesuit$) holds for $(G, \mu, \brK)$, then (FC$_\spadesuit$) holds for $(G', \mu', \brK')$. To show that (6 FC) implies (2 minute), we may therefore assume that $G$ is $\brF$-simple.

When we enlarge $K$, the condition (FC$_\spadesuit$) becomes weaker, so it is enough to check the implication that (FC$_\spadesuit$) implies $\mu$ is minute when $K$ is maximal, i.e., $\tSS - K$ is a single $\s$-orbit. In other words, $\brK$ is a maximal $\s$-stable parahoric subgroup of $G(\brF)$. Note that the parabolic subgroup $W_K$ is finite.

To complete the proof that (FC$_\spadesuit$) implies $\mu$ is minute, and thus the proof of Theorem~\ref{main}, we are now reduced to proving the following statement:
\begin{quotation}
For an $\brF$-simple and adjoint group $G$ over $F$ and maximal $\s$-stable $K\subsetneq \tSS$, (FC$_\spadesuit$) implies that $\mu$ is minute.
\end{quotation}
For the rest of this section, we assume that these assumptions are satisfied.

\subsection{} Let $K'=\s(K')$ be a union of some connected components of $K$ (view $K$ as a Dynkin diagram). We denote by $\th_{\uK'}$ the sum of all highest roots in $\Phi_{\uK'}^+$.
Let $V_{\uK'}=\Phi_{\uK'}^\vee \otimes \BR \subseteq V$ and $U_{K'} \subseteq V$ be the intersection of the affine root hyperplanes $H_i$ for $i \in K'$. Let $e^K$ be the barycenter of the facet corresponding to $\tSS-K$. Since $K$ is $\s$-stable, $\s(e^K)=e^K$.

We are interested in certain $\s$-Coxeter elements in $W_{K'}$. By definition, an element $c$ in $W_{K'}$ is a $\s$-Coxeter element if in some (equivalently, any) reduced expression of $c$ exactly one simple reflection out of each $\s$-orbit occurs.

\begin{lemma}\label{cox-ell}
Let $c$ be a $\s$-Coxeter element of $W_{K'}$. If $\l \in X_*(T)_{\G_0}$ is central on $K-K'$, i.e., $\< \l, \underline\a_i \> = 0$ for all $i\in K-K'$, then the Newton point of $t^\l c$ is central.
\end{lemma}
\begin{proof}
Set $K''=K-K'$. We have $c\s(U_K)=U_K$, $p(c \s)(V_{\uK})=V_{\uK}$ and $\BR \Phi^\vee=V_{\uK''} \oplus V_{\uK'} \oplus (U_K-e^K)$ is an orthogonal decomposition. Therefore, to compute the Newton vector, we can consider the three summands individually. Since $\l$ is central on $K''$, it is enough to look at the second and third summands.

We start by proving that $p(c \s)-\id$ is invertible on $V_{\uK'} \oplus (U_K-e^K)$. Note that $U_K=\{\sum_{j \in \tSS-K} a_j v^j; \sum_{j \in \tSS-K} a_j=1\}$, where $v^j \in \bar \aa$ is the vertex corresponding to $j$. Note that $\tSS-K$ is a single $\s$-orbit in $\tSS$ since $K$ is a maximal $\s$-stable subset of $\tSS$. Thus $e^K$ is the unique $c \s$-fixed point in $U_K$, that is, $p(c \s)-\id$ is invertible on the linear subspace $U_K-e^K$. On the other hand, since $c$ is a $\s$-Coxeter element for $K'$, $0$ is the unique $p(c \s)$-fixed point in $V_{\uK'}$, which means $p(c \s)-\id$ is also invertible on $V_{\uK'}$.

Therefore $e^K$ is the unique $c \s$-fixed point in $V_{\uK'} \oplus (U_K-e^K)$. The statement of the lemma follows; cf.~the argument in~\cite[Lemma 9.4.3]{GHKR}.
\end{proof}

\subsection{} Next we discuss the construction of certain $\s$-Coxeter elements in a finite Weyl group.

Let $(W, \BS)$ be a finite Weyl group with a diagram automorphism $\s$. We say that $(\G, {\ui})$ is a based sub-diagram if $\G$ is a sub-diagram of the Dynkin diagram of $W$ and $\ui \subseteq \G$ is a subset that contains exactly one element in each connected component of $\G$. We define
\[
c_{\G, \ui; n}=\prod_{\substack{j \in \G\\ \dist(i, j)=n\\ \text{ for some } i \in \ui}} s_j,
\]
where $\dist(i, j) \in \BN \cup \{\infty\}$ denotes the distance between two vertices $i, j$ in the graph $\G$. Since the Dynkin diagram of any finite Weyl group does not contain a circle, any two vertices with the same finite distance to a given one are not adjacent to each other. Thus the above element does not depend on the choice of the ordering. We define $$c_{\G, \ui}=c_{\G, \ui; 0} c_{\G, \ui; 1} \cdots c_{\G, \ui; |\G|-1}.$$

We say that a sub-diagram of $\BS$ is a $\s$-component of $\BS$ if it is a union of some connected components of $\BS$ such that $\s$ acts transitively on these connected components.
\begin{lemma}\label{existence-Gamma-ui}
Let $(W, \BS)$ be a finite Coxeter group and $\s$ be a diagram automorphism of $W$. Let $\ui \subseteq \BS$ such that each $\s$-component of $\BS$ contains precisely one element of $\ui$. Then there exists a based sub-diagram $(\G, \ui)$ such that $\G$ contains exactly one element in each $\s$-orbit of $\BS$. In particular, the element $c_{\G, \ui}$ is a $\s$-Coxeter element of $W$.
\end{lemma}
\begin{proof}
It suffices to consider the case that $\BS$ is connected, which follows easily by a case-by-case analysis.
\end{proof}

The following result describes the action of $c_{\G, \ui}$ on $V$.

\begin{lemma} \label{expr}
Let $W$ be a finite Weyl group and $V$ be a geometric representation of $W$. Let $(\G, \ui)$ be a based sub-diagram such that each $\s$-component of $\BS$ contains precisely one element of $\ui$. Then for any $v \in V$, $$c_{\G, \ui}(v)=v-\sum_{j \in \G} \<v, \g_j\> \a^\vee_j,$$ where $\g_j=c_{\G, \ui; |\G|-1} c_{\G, \ui; |\G|-2} \cdots c_{\G, \ui; \dist(\ui, j)+1}(\a_j)$ is a root in $\Phi(W)$.
\end{lemma}
\begin{proof}
It suffices to consider the case where $\s$ acts transitively on the set of connected components of the Dynkin diagram of $W$. In this case, $\ui=\{i\}$ for some $i \in \G$. We simply write $c_i$ for $c_{\G, \ui}$ and $c_{i, n}$ for $c_{\G, \ui; n}$. We show by induction on $n$ that $$c_{i, n} \cdots c_{i, |\G|-1}(v)=v-\sum_{j \in \G, \dist(i,j) \ge n} \<v, \g_j\> \a_j^\vee.$$ Note that
\begin{equation*}
\tag{a}
j, j' \in \G\text{ are not adjacent if } \dist(i, j)=\dist(i, j').
\end{equation*}

If $n=|\G|-1$, the statement follows from (a). Assume it holds for $n+1$. We show it holds for $n$. Set $\G(m)=\{j \in \G; \dist(i, j)=m\}$ and $\G(\ge m)=\cup_{m' \ge m} \G(m')$. For each $j \in \G(n)$ we have
\begin{align*}
\g_j &=c_{i, |\G|-1} \cdots c_{i, n+1}(\a_j)=(\prod_{j' \in \G(n+1)} s_{\g_{j'}}) (c_{i, |\G|-1} \cdots c_{i, n+2}(\a_s)) \\
\tag{b}
&=(\prod_{j' \in \G(n+1)} s_{\g_{j'}})(\a_j)=\a_j-\sum_{j' \in \G(n+1)} \<\g_{j'}^\vee, \a_j\> \g_{j'} \\ &=\a_j-\sum_{j' \in \G(\ge n+1)} \<\a_{j'}^\vee, \a_j\> \g_{j'},
\end{align*}
where the second and fifth equalities hold because $j, j' \in \G$ are not adjacent if $\dist(i, j') \ge n+2$; the third equality follows from (a). Then by induction hypothesis we have \begin{align*} & c_{i, n} \cdots c_{|\G|-1}(v)=c_{i, n} (v-\sum_{j' \in \G(\ge n+1)} \<v, \g_{j'}\> \a_{j'}^\vee) \\ &=v-\sum_{j' \in \G(\ge n+1)} \<v, \g_{j'}\> \a_{j'}^\vee-\sum_{j \in \G(n), \ j' \in \G(\ge n+1)}(\<v, \a_j\> \a_j^\vee -\<v, \g_{j'}\> \<\a_{j'}^\vee, \a_j\> \a_j^\vee) \\ &=v-\sum_{j'' \in \G(\ge n)} \<v, \g_{j''}^\vee\> \a_{j''}^\vee, \end{align*} where the first equality follows from induction hypothesis, the last one follows from (b). The lemma is proved.
\end{proof}

We now return to our usual setting.
The following result gives particular elements in ${}^K \Adm(\{\mu\})_\spadesuit$.
\begin{proposition} \label{minimal}
Let $\l \in \{\mu\}$ be a $K$-dominant cocharacter. Let $K'$ be the union of $\s$-components $C$ of $K$ such that $\l$ is noncentral on $C$. Let $(\G, \ui)$ be a based sub-diagram of $K'$ such that $\G$ contains exactly one element in each $\s$-orbit on $K'$ and that $\<\l, \ua_i\> >0$ for any $i \in \ui$. Then $t^\l c_{\ui} \in {}^K \Adm(\{\mu\})_\spadesuit$. Moreover, the Newton point of $t^\l c_{\ui}$ is central.
\end{proposition}
\begin{proof}
The "Moreover'' part follows from Lemma \ref{cox-ell}. To prove the main part, write $c_{\ui}$ as a product according to the connected components of $\G$. Since the factors commute with each other, it is enough to show that $t^\l c\in {}^K \Adm(\{\mu\})_\spadesuit$ for each factor $c$. This means that without loss of generality, we can assume $\ui=\{i\}$ for some $i \in K'$ and hence that $\G$ is connected (we will not require the condition that $\G$ contains exactly one element in each $\s$-orbit anymore).

Let $y=t^{-e^K} t^\l c_i t^{e^K}$. By \S\ref{reduce} (b), it remains to show $y \in {}^{\uK} \tW_{\uK}$. Let $c_i=s_{i_1} \cdots s_{i_r}$ be a reduced expression with each $i_j \in \G$ and $i_1=i$. Then $y=t^\l p(c_i)$. Since each $p(s_{i_j})$ is a simple reflection of $W_{\uK}$ and $p(c_i)=p(s_{i_1}) \cdots p(s_{i_r})$ is a reduced expression, to show $y \in {}^{\uK} \tW_{\uK}$, it suffices to show $\<\l, p(s_{i_1}) \cdots p(s_{i_{j-1}})(\ua_{i_j})\> \ge 1$ for $1 \le j \le r$. Notice that $i_1, \dots, i_r$ are distinct vertices and they form a connected subdiagram of $\tSS$. Thus, by the construction of $c_i$, one deduces that $$p(s_{i_1}) \cdots p(s_{i_{j-1}})(\ua_{i_j}) \in \ua_{i_1} + \BN \{\ua_{i_j}; 1 \le j \le r\}.$$ So we have $\<\l, p(s_1) \cdots p(s_{j-1})(\ua_j)\> \ge \<\l, \ua_{i_1}\> \ge 1$ as desired.
\end{proof}

\subsection{} \label{teq}
We say a triple $(\l, \G, \ui)$ is {\it permissible} (with respect to $(G, \mu, K)$) if it satisfies the assumption in Proposition \ref{minimal}. Denote by $K_{\ui}$ the union of connected components of $K$ that intersect with $\ui$. For any $j \in K$, we denote by $\o_{j, K}$ the fundamental weight associated to the root system $\Phi_{\uK}$ and we set $\o_{\CO, K}=\sum_{j \in \CO} \o_{j, K}$ for any $\s$-orbit $\CO$ in $K$. Let $p_{\uK}: V \to V_{\uK}$ be the orthogonal projection (with respect to the inner product on $V$ fixed in the beginning, Section~\ref{sec:group-theoretic-setup}). Let $v \in V$. Set $p_{\uK}(v)^\bot=v-p_{\uK}(v)$, $v_K=p_{\uK}(v-e^K)$ and $v^K=v-v_K \in U_K$. Then we have $v=p_{\uK}(v)+p_{\uK}(v)^\bot=v_K+v^K$. For $c=c_{\G, \ui}$ we have that $$t^\l c \s (v)=t^\l c \s(v^K)+p(c) p(\s)(v_K).$$ Since $c \in W_K$, we have $c(\s(v^K))=\s(v^K)$.

Let $v$ be a fixed point of $t^\l c \s$. Then $$t^{p_{\uK}(\l)} p(c) p(\s) (v_K)=v_K, \qquad t^{p_{\uK}(\l)^\bot} \s(v^K)=v^K.$$

\begin{lemma} \label{formula}
Let $\CO$ be a $\s$-orbit in $K$ and let $j$ be the unique element in $\G \cap \CO$. Then $$\<\l, \o_{\CO, K}\>=\<v_K, p(\s)\i(\g_j)\>.$$
\end{lemma}

\begin{proof}
Since $p_{\uK}(\l) + p(c) p(\s)(v_K)=v_K$, we have \begin{align*} \<\l, \o_{\CO, K}\>&=\<v_K-p(c) p(\s)(v_K), \o_{\CO, K}\> \\ &=\<v_K-p(\s)(v_K), \o_{\CO, K}\> + \sum_{\iota \in \G} \<p(\s)(v_K), \g_{\iota}\> \<\a_{\iota}^\vee, \o_{\CO, K}\> \\ &=\<p(\s)(v_K), \g_j\> = \<v_K, p(\s)\i(\g_j)\>,\end{align*} where the second equality follows from Lemma \ref{expr}.
\end{proof}

The following proposition puts strong restrictions on elements arising from a permissible triple and at the same time satisfying the fixed point condition. It will be the key to the proof that  (FC$_\spadesuit$) implies that $\mu$ is minute.

\begin{proposition}\label{mini-k}
Let $(\l, \G, \ui)$ be a permissible triple such that $t^\l c_{\G, \ui} \s$ has a fixed point $v\in\bar \aa$, then

(1) $\<v_K, \th_{\uK}\> \le 1$. Moreover, if equality holds, then we have $\<v, \ua_j\>=0$ for all $j \in \BS_0 \cap K$ such that the multiplicity of $\ua_j$ in $\th_{\uK} \in \BN \Phi_K^+$ is strictly smaller than the multiplicity of $\a_j$ in $\th$. In particular, $v^K=-\sum_{j \in \BS_0 - K} \<v_K, \a_j\> \o_j^\vee$;

(2) $\l$ is minute for $(\Phi_{\uK}, p(\s)|_{\Phi_{\uK}})$;

(3) Either $K_{\ui}$ is connected, or $K_{\ui}$ is a disjoint union of two copies of $A_1$ and $(p_{K_{\ui}}(\l), \s |_{K_{\ui}})$ equals (up to automorphism of $K$) $((\o_1^\vee, 0), {}^1 \varsigma_0)$ or $((\o_1^\vee, \o_1^\vee), {}^1 \varsigma_0)$ or $((\o_1^\vee, \o_1^\vee), \id)$, where ${}^1 \varsigma_0$ is the automorphism exchanging the two connected components of $K_{\ui}$.
\end{proposition}
\begin{proof}
We first show that
\begin{equation*}
\tag{*} \text{for any } \g \in \Phi^+_{\uK}:\ 0 \le \<v_K, \g\>.
\end{equation*}

By definition, $\<v_K, \g\>=\<v-e^K, \g\>$. To prove that $\<v_K, \g\> \ge 0$, it suffices to consider the case where $\g=\ua_i$ is a simple root in $\Phi_{\uK}$. If $i \in \BS_0$, then $\<e^K, \a_i\>=0$ and $\<v, \a_i\> \ge 0$. Otherwise, $\g=-\th$, where $\th$ is the highest root in $\Phi^+$. In this case, $\<v-e^K, -\th\>=\<e^K, \th\>-\<v, \th\>=1-\<v, \th\> \ge 0$.
(*) is proved.

Now we prove (1). If $K \subseteq \BS_0$, then $\th-\th_{\uK} \in \Phi^+$ and $\<e^K, \th_{\uK}\>=0$, where $\th$ is the highest root in $\Phi^+$. So $\<v_K, \th_{\uK}\>=\<v-e^K, \th_{\uK}\> \le 1$ since $v \in \bar \aa$. Otherwise, we have $\th_{\uK}=-\th+\d$, where $\d \in \sum_{i \in K \cap \BS_0} \BN \ua_i$. Then $\<-e^K, \th_{\uK}\>=1$ and $\<v_K, \th_{\uK}\>=1+\<v, \th_{\uK}\> \le 1$ The ``Moreover" part now follows easily.

(2) Using (*) and (1), we have $\< v_K, \g\> \le \< v_k, \th_{\uK}\>\le 1$, so
\[
0 \le \<v_K, \g\> \le 1\text{ for any } \g \in \Phi_{\uK}^+.
\]
By Lemma \ref{formula}, we have $\<\l, \o_{\CO, K}\>=\<v_K, \g\> $ for some $\g \in \Phi_{\uK}^+$. Hence $0 \le \<\l, \o_{\CO, K}\> \le 1$. Now part (2) follows from the fact that $\Phi_{\uK}$ is quasi-split with respect to $\s_0$.

(3) This follows from (1) via a case-by-case computation.
\end{proof}

Let us explain how this proposition is the key tool to prove that (FC$_\spadesuit$) implies that $\mu$ is minute. Say for simplicity that $K$ is connected (the general case will be dealt with below). Assuming that (FC$_\spadesuit$) holds in a case where $\mu$ is not minute, we have to find a $K$-dominant $\l\in\{\mu\}$ which is not central on $K$. Then there exists $i\in K$ with $\<\l, \ua_i\> >0$. By Lemma~\ref{existence-Gamma-ui}, there is a based sub-diagram $(\G, \{ i\})$ of $K$ such that $\G$ contains exactly one element in each $\s$-orbit of $K$. By Prop.~\ref{minimal}, the element $t^\l c_{\{i\}}$ lies in ${}^K\Adm(\{\mu\})_\spadesuit$ and hence $t^\l c_{\{i\}}\s$ has a fixed point in $\bar\aa$. If we can arrange things such that one of the properties (1), (2), (3) of Prop.~\ref{mini-k} fails, then we have obtained the desired contradiction. A particularly simple case is given as the following corollary.

\begin{corollary} \label{special}
Suppose that $\tSS-K$ consists of a single $\s$-fixed special vertex, then (FC$_\spadesuit$) $\Rightarrow$ $\mu$ is minute.
\end{corollary}

\begin{proof}
Under this assumption, we may assume that $K = \BS_0$. By part (2) of Prop.~\ref{mini-k} it is then enough to show that there exists a permissible triple of the form $(\mu, \Gamma, \ui)$. We choose $i$ with $\<\mu, \a_i\> > 0$, set $\ui = \{i\}$ and apply Lemma~\ref{existence-Gamma-ui} to find a suitable $\Gamma$.
\end{proof}

Finally, we prove that

\begin{proposition}\label{conn}
Assume $\tSS$ is connected. If $\mu$ is not minute, then there exists some permissible triple $(\l, \G_{\ui}, \ui)$ such that (FC)$_\spadesuit$ fails for $t^\l c_{\ui}$.
\end{proposition}

\subsection{} We first prove the Proposition for classical groups, except for small rank cases: $\tilde A_2$, $\tilde A_3$, $\tilde B_3$, $\tilde C_2$, $\tilde C_3$, $\tilde D_4$.

If $K$ is not connected, then a straightforward case-by-case analysis shows that we may choose $\l$ such that either $\l$ is non-minute on $\Phi_{\uK}$ or $\l$ is noncentral on at least two of the connected components of $K$ that are not both isomorphic to $A_1$. This contradicts Proposition \ref{mini-k}.

Now we assume that $K$ is connected. By the proof of Corollary \ref{special}, the corresponding parahoric subgroup $\brK$ is not a special maximal parahoric subgroup. Since $K$ is a maximal $\s$-stable proper subset of $\tSS$, we are in one of the following cases (up to isomorphism):

\begin{itemize}
\item $G$ is of type ${}^2 \tilde A_{n-1}$ (for $n \ge 5$), and $K$ is obtained from $\tSS$ by removing a $\s$-orbit which is of the form $A_2$. Now $K$ is of type $A_{n-2}$ and we can assume $K=\{1, 2, \cdots, n-2\}$. Since $\mu$ is non-minute, that is, $\mu \notin \{\o_1^\vee, \o_{n-1}^\vee\}$, we can choose a $K$-dominant $\l\in W_0\mu$ such that $p_{\uK}(\l) \notin \{\o_{1, K}^\vee, \o_{n-2, K}^\vee\}$ and $\l$ is not central on $K$, that is, Prop.~\ref{mini-k} (2) fails.

\item $G$ is of type ${}^2 \tilde B_n$ (for $n \ge 4$) and $K=\tSS-\{0, 1\}$ is of type $B_{n-1}$. Since $\mu \neq \o^\vee_1$, we can choose $\l$ such that $p_{\uK}(\l) \neq \o_{2, K}$ and hence Prop.~\ref{mini-k} (2) fails.

\item $G$ is of type $\tilde B_n$ or ${}^2 \tilde B_n$ (for $n \ge 4$) and $K=\tSS-\{n\}$ is of type $D_{n-1}$. Again, since $\mu \neq \o^\vee_1$, we can choose $\l$ such that $\l \neq \o_{n-1, K}^\vee$ and hence Prop.~\ref{mini-k} (2) fails.

\item $G$ is of type ${}^2 \tilde C_n$ (for $n \ge 3$) and $K=\tSS-\{0, n\}$ is of type $A_{n-1}$. We can choose $\l$ such that $p_{\uK}(\l) \neq 0$. If~\ref{mini-k} (2) holds, then $p_{\uK}(\l) \in \{\o_{1, K}^\vee, \o_{n-1, K}^\vee\}$. Without loss of generality, we assume $p_{\uK}(\l)=\o_2^\vee$. Take $(\G, \ui)=(\{2, \dots, \lfloor \frac{n}{2} \rfloor \}, \{2\})$. Then $(\l, \G, \ui)$ is a permissible triple. Let $v$ be a fixed point of $t^\l c_{\G, \ui} \s$. By Lemma \ref{formula}, one computes that $\<v_K, \th_{\uK}\>=\<v_K, \ua_{n-1}\>=1$. Therefore, Prop.~\ref{mini-k} (1) fails since the multiplicity of $\ua_{n-1}$ is strictly smaller than the multiplicity of $\a_{n-1}$ in $\th$.

\item $G$ is of type ${}^4 \tilde D_n$ (for $n \ge 5$) and $K=\tSS-\{0, 1, n-1, n\}$ is of type $A_{n-3}$. Then Prop.~\ref{mini-k} (1) fails similarly as the ${}^2 \tilde C_n$ case above.

\item $G$ is of type ${}^2\tilde D''_n$ (for $n \ge 5$) and $K=\tSS-\{0, n\}$ is of type $A_{n-1}$. Again, Prop.~\ref{mini-k} (1) fails similarly as the ${}^2 \tilde C_n$ case above.

\item $G$ is of type ${}^2\tilde D_n$ or ${}^2\tilde D_n'$ (for $n \ge 5$) and $K=\tSS-\{0, 1\}$ is of type $D_{n-1}$. Since $\mu \neq \o^\vee_1$, we can choose $\l$ such that $\l \neq \o_{2, K}^\vee$ and hence Prop.~\ref{mini-k} (2) fails.
\end{itemize}

\subsection{} The basic idea for the small rank classical groups and the exceptional groups is similar, but involves more case-by-case analysis and direct constructions of the permissible triples that do not satisfy (FC)$_\spadesuit$. We provide the details for type $\tilde A_3$, type $\tilde D_4$ and type $\tilde E_6$.

\subsubsection{Type $\tilde A_3$}

If $\s \in \{\t_1, \t_1\i\}$, then $\s$ acts transitively on $\tSS$ and hence $K=\emptyset$. Assume (FC) holds for all elements attached to a permissible triple $(\l, \emptyset, \emptyset)$. Then each element of $\Adm(\{\mu\})$ satisfies the fixed point condition. So $B(G, \{\mu\})$ consists of a single element. By \cite[6. 11]{kottwitz-isoII}, $\mu \in \{\o_1^\vee, \o_{n-1}^\vee\}$, contradicting the assumption that $\mu$ is non-minute.

If $\s=\t_1^2$, we can take $K=\{1, 3\}$. If $\mu$ is non-minuscule, we can take $\l$ such that $\<\l, \a_1\> \ge 2$ and Prop.~\ref{mini-k} (3) fails for the permissible triple $(\l, \{1\}, \{1\})$. If $\mu=\o_1^\vee$, then (FC) fails for $(\l, \G \ui)=(s_1(\mu), \{2\}, \{2\})$ via direct computation.

If $\s=\varsigma_0$, then $\mu$ is non-minuscule since $\mu$ is non-minute. So we can assume $\<\l, \a_i\> \ge 2$ for some $j \in K$. If $K=\{1, 0, 3\}$, then Prop.~\ref{mini-k} (2) fails. Assume $K=\{0, 2\}$. Let $(\l, \G, \ui)$ be a permissible triple and let $v$ be a fixed point of $t^\l c_{\G, \ui} \s$. One computes that $\<v_K, \th_{\uK}\> \ge 1$ since $\<\l, \a_i\> \ge 2$. If Prop.~\ref{mini-k} (1) holds, then $\<v_K, \th_{\uK}\> = 1$, that is, $(\l, \a_j)=2$ and $\<\l, \a_{j'}\>=0$, where $j'$ is the other element of $K$ different from $j$. In this case, $(\G, \ui)=(\{j\}, \{j\})$. By the ``Moreover'' part of Prop.~\ref{mini-k} (1), we see that $\s(v^K)=v^K$ and hence $\l\in V_K$. So $\mu=-\ua_0^\vee$. Now we take $(\l, \G, \ui)=(s_2 s_1(\mu), \{0, 2\}, \{0, 2\})$. Again we deduce that $\l \in V_K$, which is impossible.

Assume $\s=\t_1 \varsigma_0$. If $\mu$ is non-minuscule, (FC) fails for some permissible triple similarly as in the $\s=\varsigma_0$ case. It remains to consider the case $\mu=\o_2^\vee$. Assume $K=\{0, 1\}$. Then (FC) fails for $(\l, \G, \ui)=(s_1 s_0(\mu), \{1\}, \{1\})$ via direct computation.

\subsubsection{Type $\tilde D_4$.} Let $(\l, \G, \ui)$ be a permissible triple such that $\l$ is noncentral on $\Phi_{\uK}$ and let $v$ be a fixed point of $t^\l c_{\G, \ui}$.

If $\s$ is of order $4$. Then $\{0, 1, 3, 4\}$ is a $\s$-orbit. Suppose $K=\{2\}$ and (FC) holds for $(\l, \G, \ui)$. Then $\<\l, \a_2\> \ge 1$ and $\<v, \a_2^\vee\>=\<v_K, \a_2^\vee\>=\frac{1}{2} \<\l, \a_2\> \le 1$. If the equality holds, then Prop.~\ref{mini-k} (1) fails since the multiplicity of $\a_2$ in $\th$ is two. So $\<\l, \a_2\>=1$ and hence $\<v^K, \a_j^\vee\>=\frac{1}{4}$ for $j \in \{1,3, 4\}$. Now $\s(v^K)=v^K$ and hence $\l \in V_K$, which is impossible. Suppose $K=\{0, 1, 3, 4\}$. Then Prop.~\ref{mini-k} (3) fails directly.

If $\s$ is of order $3$, we may assume $\{1, 3, 4\}$ is a $\s$-orbit.  If $K$ equals $\{1, 3, 4\}$ or $\{0, 1, 3, 4\}$, we can assume further that $\l$ is noncentral on $\{1, 3, 4\}$, then Prop.~\ref{mini-k} (3) fails directly. Assume $K=\{0, 2\}$ and (FC) holds holds for $(\l, \G, \ui)$. Then $\l$ is minuscule (or minute) on $K$ by Prop.~\ref{mini-k} (2). One computes by Lemma \ref{formula} that $\<v_K, \a_2^\vee\>=\<v_K, \a_0^\vee\>=\frac{1}{3}$. Since the multiplicity of $\a_2$ in $\th$ is two, we have by Prop.~\ref{mini-k} (1) that $\<v, \a_j\>=0$ and hence $\<v^K, \a_j^\vee\>=\frac{1}{3}$ for $j \in \{1, 3, 4\}$. So $\s(v^K)=v^K$ and hence $\l \in V_K$, which is impossible. The case $K=\{2\}$ can be handled similarly.

\subsubsection{Type $\tilde E_6$}
Case (1): $\s=\id$.

If $K=\tSS-\{j\}$ for $j \in \{0, 1, 6\}$, then Prop.~\ref{mini-k} (2) fails by Corollary \ref{special}.

If $K=\tSS-\{j\}$ for $j \in \{2, 3, 5\}$, we can assume $j=2$ and claim that $\l$ such that $\l$ is non-minute for $K'=\tSS-\{0, 2\}$ and hence Prop.~\ref{mini-k} (2) fails. Indeed, if $\l$ is minute for $K'$, that is, $p_{\uK}(\l) \in \{\o_{1, K'}^\vee, \o_{6, K'}^\vee\}$, then $\<\l, \a_2\> \le -1$ since $\mu \neq 0$. Then the $K$-dominant cocharacter in the $W_K$-orbit of $s_2(\l)$ is non-minute for $K'$ and the claim is verified.

If $K=\tSS-\{4\}$, then, by a similar argument as in $K=\tSS-\{2\}$, we can assume $\l$ is noncentral on at least two connected components of $K$ and Prop.~\ref{mini-k} (3) fails.

Case (2): $\s=\varsigma_0$.

If $K$ equals $K=\tSS-\{0\}$ or $K=\tSS-\{2\}$ or $K=\tSS-\{4\}$, the arguments are similar as in Case (1).

If $K=\tSS-\{1, 6\}$, we show that we can take $\l$ such that $\l$ is non-minute for $K$ and hence Prop.~\ref{mini-k} (2) fails. Assume $\l$ is minute for $K$. Notice that one of $\{1, 6\}$, say $1$, satisfies $\<\l, \a_1\> \le 1$ since $\mu \neq 0$. Then the $K$-dominant cocharacter in the $W_K$-orbit of $s_1(\l)$ is non-minute for $K$.

If $K=\tSS-\{3, 5\}$, then Prop.~\ref{mini-k} (2) or (3) fails since we can assume $\l$ is either non-minute on $(\Phi_{\uK}, p(\s)|_{\Phi_{\uK}})$ or noncentral on both of the two $\s$-components of $K$.

Case (3): $\s=\t_1$.

If $K$ equals $\tSS-\{0, 1, 6\}$ or $\tSS-\{2, 3, 5\}$, we can assume $\l$ is noncentral on $\Phi_{\uK}$. Then Prop.~\ref{mini-k} (2)/(3) fails since $\l$ is non-minute on $K$.

If $K=\tSS-\{4\}$, then Prop.~\ref{mini-k} (3) fails similarly as the $K=\tSS-\{4\}$ case in Case (1).

\section{Shimura varieties}
\label{sec:shimura}

\subsection{Setup}
\label{shimura-setup}
Let $({\bf G}, \{h\})$ be a Shimura datum, fix a prime number $p$, and let ${\bf K}={\bf K}^p{\bf K}_p$  be an open compact subgroup of ${\mathbf G}(\BA_f)$, where ${\mathbf K}^p \subseteq {\mathbf G}(\BA^p_f)$ is a sufficiently small open compact subgroup, and ${\bf K}_p \subseteq {\bf G}(\BQ_p)$ is a parahoric subgroup. Let $\mathrm{Sh}_{\mathbf K}=\mathrm{Sh}({\mathbf G}, \{h\})_{\mathbf K}$ be the corresponding Shimura variety. It is a quasi-projective variety defined over the Shimura field $\bE$.
Let $O_E$ be the ring of integers of the completion $E$ of $\bE$ at a place ${\bf p}$ above $p$.

We assume that the axioms in the paper \cite{HR} by Rapoport and the second-named author are satisfied. In particular, we dispose of an integral model ${\mathbf S}_{\mathbf K}$ of the Shimura variety over $O_E$.
Let $Sh_{\bK}={\mathbf S}_{\bK}\times_{\Spec O_{E}}\Spec  \kappa_E$ be the special fiber of ${\mathbf S}_{\bK}$.

\begin{remark}
The axioms of~\cite{HR} are known to hold in the unramified PEL case for $p > 2$, and for odd ramified unitary groups. See the joint work of the second-named author and Zhou \cite{HZ}. Most of the axioms are also verified for Shimura varieties of Hodge type by Zhou \cite{Zhou}.
\end{remark}

We write $F=\BQ_p$, $G = {\mathbf G}\otimes_{\BQ}\BQ_p$, and $\CK = {\bf K}_p$ to make the connection with our previous setting. After making some group-theoretic choices as in Section \ref{sec:group-theoretic-setup}, $\CK$ corresponds to a subset $K\subseteq \tSS$, and we will write $Sh_K$ instead of $Sh_{\bK}$ below.  We often identify $Sh_K$ with its set of $\kk$-valued points $Sh_K(\kk)$.
From the Shimura datum, we obtain a conjugacy class $\{ \mu\}$ of cocharacters (we follow the conventions of \cite{HR}). Many of the axioms are related to the following commutative diagram:
\begin{equation}\label{diagram-hr-axioms}
\begin{gathered}
\xymatrix{ & & \brK \backslash G(\breve\BQ_p)/\brK \\  Sh_K  \ar[r]^-{\Upsilon_K} \ar@/^1pc/[urr]^{ \lambda_K } \ar@/_1pc/[drr]_{ \delta_K } & G(\breve\BQ_p)/\brK_\sigma \ar[ur]_{\ell_K} \ar[dr]^{d_K} & \\ & & B(G)}.
\end{gathered}
\end{equation}
Here by $\brK_\s$ we indicate that $\brK$ acts by $\s$-conjugation.

We then have
\begin{align*}
& \Upsilon_K(Sh_K)=\cup_{w \in \Adm(\{\mu\})} \brK w \brK/\brK_\sigma\\
& \l_K(Sh_K)=\brK \backslash \cup_{w \in \Adm(\{\mu\})} \brK w \brK/\brK=W_K \backslash \Adm^K(\{\mu\})/W_K,\\
& \d_K(Sh_K)=B(G, \{\mu\}).
\end{align*}

For any $[b] \in B(G, \{\mu\})$, the \emph{Newton stratum} $\mathit S_{K, [b]}$ is defined to be the fiber of $\d_K$ over $[b]$.

For any $w \in W_K \backslash \Adm^K(\{\mu\})/W_K$, the \emph{Kottwitz-Rapoport stratum (KR stratum)} $KR_{K, w}$ is defined to be the fiber of $\l_K$ over $w$.

There is another stratification of $Sh_K$ introduced in \cite{HR}. For any $w \in \Adm(\{\mu\}) \cap {}^K \tW$, we set
\[
EKOR_{K, w}=\Upsilon_K \i(\brK_\sigma (\brI w \brI)/\brK_\sigma)
\]
and call it the {\it Ekedahl-Kottwitz-Oort-Rapoport stratum (EKOR-stratum)} of $Sh_K $ attached to $w \in {}^K \tW$. The EKOR stratification is finer than the KR stratification. If $G$ is unramified and $K$ is hyperspecial, the EKOR stratification coincides with the Ekedahl-Oort stratification in \cite{Vi}.

\subsection{Shimura varieties, RZ spaces and affine Deligne-Lusztig varieties}
Note that although no affine Deligne-Lusztig varieties appear in the commutative diagram in the previous section, using the restriction of the Lang map $G(\breve{\BQ}_p) \rightarrow G(\breve \BQ_p)/\brK_\s$, $g\mapsto g^{-1}b\s(g)$, to $X(\mu, b)_K$, we obtain a bijection
\[
\JJ_b\backslash X(\mu, b)_K(\kk) \isoarrow
d_K^{-1}([b]) \cap \ell_K^{-1}(\Adm^K(\{\mu\})).
\]
In particular, $\Upsilon_K$ restricts to a map from the Newton stratum for $b$ to $\JJ_b\backslash X(\mu, b)_K(\kk)$:
\begin{equation}\label{NewtonStratumVsADLV}
\Upsilon_K(\d_K\i([b])) = d_K\i([b])\cap \Im\Upsilon_K = d_K\i([b])\cap\ell_K\i(\Adm^K(\{\mu\}))
= \JJ_b\backslash X(\mu, b)_K(\kk).
\end{equation}

Another (but similar) way to relate Newton strata to affine Deligne-Lusztig varieties is via Rapoport-Zink spaces. In the following discussion, we always use the ``flat RZ space'' (whose special fibre is described, by the coherence conjecture as proved by Zhu, by the $\mu$-admissible set), rather than the ``naive RZ space''. By the uniformization theorem (by Rapoport-Zink in the PEL case; by Kim~\cite{kim} in the case of Shimura varieties of Hodge type), the Rapoport-Zink space attached to the pertaining data is related in an explicit way to the basic locus of the Shimura variety. Cf. \cite{Vollaard}, \cite{Vollaard-Wedhorn} for an example of how to make this explicit.

For the non-basic Newton strata, the connection is more complicated: To describe the full Newton stratum, the leaves in the sense of Oort have to be taken into account as well. In the case of a hyperspecial level structure, the Newton stratum has been shown to have an ``almost-product structure'', i.e., there is a surjective finite-to-finite correspondence between the Newton stratum and the product of the central leaf and a truncated RZ space. See \cite{Mantovan:AlmostProduct}, \cite{Hamacher:AlmostProduct}, \cite{CaraianiScholze}.

Consider the following condition which can be thought of as an analogue of (part of) the axioms in \cite{HR}, in the setting of RZ spaces:

\bigskip
($\Diamond$) \hspace{.5cm}
There is an isomorphism of perfect schemes
\[
\CM(G, \mu, b)_{K, \kk}^{p^{-\infty}} \cong X(\mu, b)_K.
\]
\bigskip

In \cite[Prop. 0.4]{Zhu}, this is proved in the case of hyperspecial level structure. We expect that the arguments there can be extended to the case of RZ spaces attached to a PEL-datum (using the coherence conjecture proved by Zhu which ensures that the special fiber of the RZ space ``corresponds to'' the admissible set figuring in the definition of $X(\mu, b)_K$), and that things can be arranged such that via the uniformization theorem of Rapoport and Zink, one obtains a commutative diagram relating the maps ($\Diamond$) and the map from the Newton stratum for $b$ to $\JJ_b\backslash X(\mu, b)_K$ induced by $\Upsilon_K$.

\subsection{Hodge-Newton decomposition for RZ spaces}
\label{subsec:HN-for-RZ}

To make a precise statement, let us fix the following notation: Assume that $(G, b, \mu)$ arises from a PEL-datum, fix a level structure $K\subseteq \tSS$, $\s(K)=K$, and let $\CM=\CM(G, \mu, b)_K$ be the corresponding RZ space.

For a semi-standard Levi subgroup $M$, a conjugacy class $\{\mu_M\}$ inducing $\{\mu\}$ and an element $b_M\in [b]\cap M(\brF)$, we have a closed immersion
\[
\CM(M, \mu_M, b_M)_{K_M} \rightarrow \CM(G, \mu, b)_K
\]
(after base change to the ring of integers of the compositum of the reflex fields involved). See~\cite[Example 5.2 (iii)]{RV}. When regarded as a moduli space of $p$-divisible groups with additional structure, the image of this immersion is the locus where one has a product decomposition of the type ``specified by $M$''. Note that in~\cite{RV}, it is assumed that the framing object has a product decomposition, i.e., lies in the image of the map; this means that the closed immersion above should be thought of as the composition of the map in \cite{RV} and the automorphism of $\CM$ given by the change of framing object, analogously to the map~\eqref{map-from-levi}. Similarly as for affine Deligne-Lusztig varieties, while the map itself depends on these choices, the image does not.

We will require below that the condition ($\Diamond$) is satisfied for $G$ and $M$, and that these isomorphisms give rise to a commutative diagram with the map in~\eqref{map-from-levi}. This is true in the case of~\cite[Prop.~0.4]{Zhu}.

From the Hodge-Newton decomposition for affine Deligne-Lusztig varieties, Theorem~\ref{HN-dec}, we obtain

\begin{theorem}
In the above situation, assume that ($\Diamond$) holds, compatibly with passing to Levi subgroups, and the conditions in Section~\ref{HN-setup} hold, i.e., the pair $(\mu, b)$ is HN decomposable with respect to some Levi subgroup $M$. Then we have a decomposition
\[
\CM(G, \mu, b)^{\rm red}_{K, \kk} \cong \bigsqcup_{P'=M'N'\in \fkP^\s / W^\s_K } \Im\left(\CM(M', \mu_{P'}, b_{P'})_{K, \kk} \rightarrow \CM(G, \mu, b)_{K, \kk}\right)^{\rm red}.
\]
as a disjoint union of open and closed subschemes.
\end{theorem}

\begin{proof}
The morphisms described above induce a morphism from the union of the RZ spaces attached to the Levi subgroups to the space on the left hand side. The properties that this is a surjection, that the union is disjoint and that the subsets are open and closed can be checked after passing to the perfection. There the desired statement follows directly from Theorem~\ref{HN-dec}, using property ($\Diamond$).
\end{proof}

We expect that the result also holds for RZ spaces of Hodge type; to establish this, one needs to define the closed immersions of spaces coming from Levi subgroups, and check property ($\Diamond$). Furthermore, it seems reasonable to hope that this result holds not only on the level of the underlying reduced subschemes, but also on the level of the corresponding adic spaces, i.e., for the generic fibres. Results of this type have been proved by several people, cf.~Mantovan~\cite{mantovan}, Shen~\cite{Shen}, Hong~\cite{hong}, \cite{hong2}. All the cited papers deal with hyperspecial level structure only, however.

\subsection{The basic Newton stratum} Recall that we denote by $[b_0]$ the unique basic $\sigma$-conjugacy class in $B(G, \{\mu\})$.

\begin{definition}\label{def:EKORtype}
We say that the basic locus $\mathit S_{K, [b_0]}$ is of EKOR type if it is a union of EKOR strata.
\end{definition}

\subsection{The main theorem for Shimura varieties}

\begin{theorem}\label{main'}
If $(G, \{\mu\}, \brK)$ arises from a Shimura datum as in Section~\ref{shimura-setup}, and $G$ is quasi-simple over $F$, then the conditions (1) to (5) in Theorem~\ref{main} are equivalent to the following equivalent conditions:
\begin{enumerate}[label=(\arabic*\,$'$), start=3]
\item
All non-basic Newton strata consist of only finitely many leaves, i.e., in the notation of \cite{HR}, as recalled in Section~\ref{shimura-setup}: For all non-basic $[b]\in B(G, \{\mu\})$,
\[
\Upsilon_K(\delta_K^{-1}([b])) \text{\ is a finite set;}
\]
\item
Any EKOR stratum of $Sh_K$ is contained in a single Newton stratum;
\item
The basic locus $\mathit S_{K, [b_0]}$ is of EKOR type (Def.~\ref{def:EKORtype}).
\end{enumerate}
In particular, the validity of these conditions only depends on the choice of $(G, \{\mu\})$ and is independent of the choice of the parahoric subgroup $\CK$.
\end{theorem}

We prove the theorem by proving that each of the conditions listed above is equivalent to the group theoretic version, i.e., the condition in Theorem~\ref{main} with the corresponding label.

\subsection{Newton strata versus leaves: (3) $\Longleftrightarrow$ (3$'$)}
\label{subsec:3-equiv-3S}

In view of~\eqref{NewtonStratumVsADLV}, the equivalence of (3) and (3$'$) follows from the following

\begin{proposition}
The space $X(\mu, b)_K$ has dimension $0$ if and only if the set $\JJ_b\backslash X(\mu, b)_K$ is finite.
\end{proposition}

\begin{proof}

First note that $\dim X(\mu, b)_K=0$ if and only if for all $w$, $X(\mu, b)_K \cap \brK w \brK/\brK$ is finite. If this holds, then the main result of \cite{RZ:indag} shows that $\JJ_b\backslash X(\mu, b)_K(\kk)$ is a finite set. For the converse, the assumption says that $X(\mu, b)_K(\kk)$ is a finite union of $\JJ_b$-orbits. But every $\JJ_b$-orbit intersects a given double coset $\brK w \brK/\brK$ in only finitely many points: In fact, $\JJ_b$ is the set of $F$-points of an algebraic group whose $\brF$-points are by definition a subgroup of $G(\brF)$, and the action of $\JJ_b$ is the action induced by this inclusion. Let us denote this algebraic group by $\mathbf J$ just in this proof: $\JJ_b = \mathbf J(F)$, $\mathbf J(\brF)\subseteq G(\brF)$. (More precisely, $\mathbf J$ is the inner form of a Levi subgroup of $G$.) We can then identify $\mathbf J(F')$ with a subgroup of $G(F')$ for some finite unramified extension $F'/F$, and hence --- passing to the loop groups --- the action of $\mathbf J(\brF)$ on $G(\brF)/\brK$ is defined over the residue class field of $F'$, a finite extension of the residue class field of $F$. This implies that all points in a fixed $\JJ_b$-orbit inside $G(\brF)/\brK$ are fixed by some finite power of Frobenius. But since $\brK w \brK/\brK$ is a finite-dimensional variety over a finite field, it contains only finitely many such points.
% Cf.~Prop.~\ref{zero-diml-adlv}.

Alternatively, one could phrase the proof in terms of the stack quotient $[\JJ_b\backslash X(\mu, b)_K]$.
\end{proof}

\subsection{Compatibility of EKOR and Newton stratifications: (4) $\Longleftrightarrow$ (4$'$), and (5) $\Longleftrightarrow$ (5$'$)}
\label{subsec:4-equiv-4S}
\label{subsec:5-equiv-5S}

Both equivalences (4) $\Longleftrightarrow$ (4$'$), and (5) $\Longleftrightarrow$ (5$'$) follow immediately from the following lemma. Note that for (5) we use the ``technical'' version in Thm.~\ref{main} rather than the informal version stated in the introduction; cf.~Section~\ref{subsec:4-implies-5} for the connection.

\begin{lemma}
Let $w\in \Adm(\{\mu\})\cap {}^K\tW$, and let $[b]\in B(G,\mu)$. Then $X_w(b) \ne \emptyset$ if and only if the EKOR stratum $EKOR_{K, w}$ and the Newton stratum $S_{K, [b]}$ have non-empty intersection.
\end{lemma}

\begin{proof}
This is a straight-forward consequence of the axioms on the Shimura variety, compare especially the commutative diagram~\eqref{diagram-hr-axioms}: We have $X_w(b)\ne\emptyset$ if and only if $\brI w\brI\cap [b]\ne \emptyset$. This can be rephrased by saying that $\brK_\s(\brI w\brI)/\brK_\s \cap d_K\i([b])\ne \emptyset$. Since $\brK_\s(\brI w\brI)/\brK_\s\subseteq \Im\Upsilon_K$, this is equivalent to the condition that $\Upsilon_K\i(\brK_\s(\brI w\brI)/\brK_\s) \cap \d_K\i([b])\ne\emptyset$, as desired.
\end{proof}

\end{document}